\documentclass[fleqn,a4paper, times,11pt]{amsart}
\usepackage{amscd,amsthm,amsfonts,latexsym,amssymb}
\usepackage{graphicx}

\usepackage{bez123,calc,curves,ebezier,epic,eepic,graphicx,multiply,rotating}

\newtheorem{theorem}{Theorem} [section]
\newtheorem{lemma}[theorem]{Lemma}
\newtheorem{corollary}[theorem]{Corollary}

\newtheorem{definition}[theorem]{Definition}
\newtheorem{remark}[theorem]{Remark}

\newcommand{\dd}{{\mathrm d}}  

\newcommand{\RR}{{\mathbb R}}  

\newcommand{\Hh}{{\mathcal H}}  
\newcommand{\Vv}{{\mathcal V}}  

\newcommand{\Ll}{{\mathcal L}}
\newcommand{\Ss}{{\mathcal S}}
\newcommand{\Oo}{{\mathcal O}}
\newcommand{\Ff}{{\mathcal F}}

\newcommand{\ra}{\rightarrow}

\newcommand{\pa}{\partial}

\renewcommand{\phi}{\varphi}
\newcommand{\la}{\lambda}

\newcommand{\na}{\nabla}
\newcommand{\al}{\alpha}
\newcommand{\be}{\beta}
\newcommand{\ga}{\gamma}
\newcommand{\Ga}{\Gamma}
\newcommand{\ve}{\varepsilon}
\newcommand{\si}{\sigma}

\newcommand{\om}{\omega}

\newcommand{\ta}{\theta}

\newcommand{\wt}{\widetilde} 
\newcommand{\ov}{\overline} 

\newcommand{\Tr}{\mbox{\rm Tr}\,}

\newcommand{\Ric}{\mbox{\rm Ric}\,}
\newcommand{\dv}{\mbox{\rm div}\,}

\newcommand{\grad}{\mbox{\rm grad}\,}

\newcommand{\n}{\noindent}

\begin{document}

\title{Four-dimensional Einstein metrics from biconformal deformations}
\author{Paul Baird} 
\address{Laboratoire de Math\'ematiques de Bretagne Atlantique UMR 6205 \\
Universit\'e de Brest \\
6 av.\ Victor Le Gorgeu -- CS 93837 \\
29238 Brest Cedex, France}
\email{paul.baird@univ-brest.fr}
\author{Jade Ventura}
\address{Institute of Mathematics \\ University of the Philippines Diliman \\
C.P. Garcia Ave. \\ 1101 Quezon City \\ Philippines} 
\email{jventura@math.upd.edu.ph}



\thanks{The second author acknowledges the support of the Philippine Commission on Higher Education and Campus France. The images were created with GeoGebra by Dr. Adrian Jannetta - https://www.geogebra.org/m/W7dAdgqc}

\begin{abstract} 
Biconformal deformations take place in the presence of a conformal foliation, deforming by different factors tangent to and orthogonal to the foliation. Four-manifolds endowed with a conformal foliation by surfaces present a natural context to put into effect this process. We develop the tools to calculate the transformation of the Ricci curvature under such deformations and apply our method to construct Einstein $4$-manifolds. One particular family of examples have ends that collapse asymptotically to $\RR^2$. 
\end{abstract}

\maketitle

\section{Introduction}  \n A smooth Riemannian manifold $(M,g)$ is said to be \emph{Einstein} if its Ricci curvature satisfies $\Ric = Ag$ for some constant $A$. D. Hilbert showed how Einstein metrics arise from the variational problem of extremizing scalar curvature \cite{Hi}. The relation between scalar curvature and conformal transformations has been explored by analysts over the latter part of the last century. The Yamabe problem is to determine the existence of a metric of constant scalar curvature in a conformal class \cite{Ya}. There have been important contributions by various authors and the problem was completely solved positively in the compact case by R. Schoen \cite{Sc}; for a survey see the notes of Hebey \cite{He2}. 

Conformal transformations are not in general sufficiently discerning to find Einstein metrics.  For example, although any manifold admits a Riemannian metric, on a compact manifold, there is a topological obstruction to the existence of an Einstein metric, known as the Hitchin-Thorpe inequality \cite{Th, Hi, Be}, whereas there always exist constant scalar curvature metrics. Biconformal deformations on the other hand, have an additional parameter that allows greater freedom to satisfy the Einstein equations.   

A biconformal deformation of a Riemannian manifold $(M,g)$ (see below) takes place in the presence of a conformal foliation. A foliation $\Ff$ is \emph{conformal} if Lie transport along the leaves of the normal space is conformal, specifically, if we set $T\Ff$ to be the tangent space to the leaves and $N\Ff$ the normal space, there exists a mapping $a:T\Ff \ra \RR$, linear at each point, such that 
$$
(\Ll_Ug)(X,Y) = a(U)g(X,Y) \qquad (\forall U \in T\Ff \,, \forall X,Y \in N\Ff )\,.
$$
Conformal foliations are intimately related to semi-conformal mappings. 

A mapping $\phi : (M^m, g) \ra (N^n, h)$ is \emph{semi-conformal} if at each point where its derivative is non-zero, it is surjective and conformal (and so homothetic) on the complement of its kernel. Specifically, at each $x \in M$ where $\dd \phi_x \neq 0$, the derivative is surjective and there exists a real number $\la (x) >0$ such that
$$
\phi^*h(X,Y) = \la(x)^2 g(X,Y) \qquad (\forall X,Y \in (\ker \dd \phi_x)^{\perp})\,.
$$
Extending $\la$ to be zero at points $x$ where $\dd \phi_x = 0$, determines a continuous function $\la : M \ra \RR (\geq 0)$, smooth away from critical points, called the \emph{dilation} of $\phi$.  In \cite{Ba-Wo}, it is shown that if $\phi : (M^m,g) \ra (N^n,h)$ is a semi-conformal submersion, then its fibres form a conformal foliation; conversly, if $\Ff$ is a conformal foliation on $(M^m,g)$ and $\psi : W \subset M \ra \RR^n \times \RR^{m-n}$ is a local foliated chart, then there is a conformal metric on the leaf space $N$ of $\Ff\vert_W$ with respect to which the natural projection $\phi : W \ra N$ is a semi-conformal submersion.  The relation between $a$ above and the dilation $\la$ is given by $a = -2 \dd \ln \la \vert_{\Vv}$, where $\Vv = T\Ff = \ker \dd \phi$ \cite{Ba-Wo}.   

A mapping $\phi : (M^n, g) \ra (N^n, h)$ is a \emph{harmonic morphism} if it pulls back local harmonic functions to local harmonic functions.  The fundamental characterization of a harmonic moprhism was obtained independently by B. Fuglede (1978) \cite{Fu}, T. Ishihara (1979) \cite{Is} and A. Bernard, E. Campbell and A. M. Davie (1979) \cite{BCD}, as those mappings that are both harmonic and semi-conformal.  In the case when $m = 4$ and $n = 2$, a strong connection between harmonic morphisms and anti-self-dual Einstein $4$-manifolds was established by J. C. Wood \cite{Wo} and extended by M. Ville \cite{Vi}. This derives from the application of Riemannian twistor theory to the almost Hermitian structure $J = (J^{\Hh}, J^{\Vv})$ associated to $\phi$ when $M$ and $N$ are both oriented, where $J^{\Hh}$ (resp. $J^{\Vv}$) is equal to rotation through $+\pi /2$ in the horizontal (resp. vertical) space.  In the case of $1$-dimensional fibres, a characterization of harmonic morphisms $\phi : (M^4, g) \ra (N^3, h)$ between oriented manifolds with $M^4$ Einstein was obtained by R. Pantilie \cite{Pa}.  It is to be noted that the conditions for a harmonic morphism is an over-determined system of PDEs which is therefore restrictive on the geometric structures that support such mappings. One of our aims in this article is to remove the condition \emph{harmonic} allowing for the application of a natural metric deformation that generalizes conformal. 

Let $\phi : (M^n, g) \ra (N^n, h)$ be a semi-conformal submersion between Riemannian manifolds. Then the metric $g$ decomposes into the sum $g = g^{H} + g^{V}$ of its horizontal and vertical components. A \emph{biconformal deformation} of $g$ is a metric 
$$
\wt{g} = \frac{g^{H}}{\si^2} + \frac{g^{V}}{\rho^2}\,,
$$
where $\si , \rho : M \ra \RR$ are smooth positive functions. Note that the deformation is conformal if and only if $\si \equiv \rho$. We could equally define a biconformal deformation with respect to a conformal foliation. Such deformations preserve the semi-conformality of $\phi$.  

The idea to use biconformal deformations to construct Einstein metrics is founded on the possibility of obtaining a suitable expression for the Ricci curvature in terms of parameters of the semi-conformal map: its dilation, second fundamental form of its fibres and integrablity form associated to the horizontal distribution. Indeed, when the mapping is a harmonic morphism with $1$-dimensional fibres, an elegant expression can be derived \cite{Ba-Wo, Ba-Da}. This was exploited by L. Danielo to construct Einstein metrics in dimension $4$ by biconformally deforming the metric with respect to a harmonic morphism to a $3$-manifold, with the deformation restricted to preserve harmonicity \cite{Da-1, Da-2}. As noted, this is a limited situation because of the overdetermined system \emph{harmonic} plus \emph{semi-conformal}.  

The main obstacle to using semi-conformal mappings, giving complete freedom in the choice of $\si$ and $\rho$, is the computation of the Ricci curvature. For a $3$-manifold, this is possible and was exploited by the first author and E. Ghandour to construct Ricci solitons in dimension $3$ \cite{Ba-Gh}. For a semi-conformal map $\phi : (M^4, g) \ra (N^2, h)$, the computation of the Ricci curvature is challenging, however, as noted above, this is the natural context for Einstein metrics on $4$-manifolds due to the associated almost Hermitian structure $J$.  Indeed, the fibres of $\phi$ form a family of pseudo-holomorphic curves in the sense of Gromov, with their significance in the study of $4$-manifolds \cite{Gr}.  

In this article we achieve a computation of the Ricci curvature associated to a semi-conformal submersion $\phi : (M^4, g) \ra (N^2, h)$ (see \S\ref{sec:ricci}) and use it to construct Einstein metrics by biconformal deformation associated to orthogonal projection $\RR^4 \ra \RR^2$. Amongst the examples produced is a family of complete Einstein metrics of negative Ricci curvature with each member having an $\RR^2$-end (Theorem \ref{thm:ends}). The term \emph{end} is used loosely here to refer to a component of the exterior of a family of exaustive subsets (not compact) that collapses to $\RR^2$.  

In \S\ref{sec:first}, we calculate the connexion coefficients associated to a semi-conformal submersion $\phi : (M^4, g) \ra (N^2, h)$. We exploit these formulae in \S\ref{sec:ricci} to deduce expressions for the Ricci curvature in terms of the geometric parameters associated to $\phi$ referred to above (now with $J$ included). In \S\ref{sec:bic-def}, we obtain expressions for how these quantities change under biconformal transformation. These are then applied in \S\ref{sec:can-proj} to orthogonal projection $\RR^4 \ra \RR^2$, to deduce partial differential equations for an Einstein metric in terms of the parameters $\si$ and $\rho$. In general these are challenging to solve, but special cases yield interesting and possibly new $4$-dimensional Einstein metrics.

\section{Connection coefficients associated to a semi-conformal submersion} \label{sec:first}

\n  Let $\phi : (M^4, g) \ra (N^2, h)$ be a semi-conformal submersion between oriented Riemannian manifolds with dilation $\la : M \ra \RR^+$.  The coefficients of the Levi-Civita connection with respect to an adapted orthonormal frame field will be expressed in terms of the dilation, the mean-curvature of the fibres and an integrability form associated to the horizontal distribution. 

 Let $\{ f_1, f_2\}$ be a positive orthonormal frame on $N^2$.  Then in general $\na f_1 = \rho_{12} f_2$ and $\na f_2 = \rho_{21} f_2$ where $\rho_{12} = - \rho_{21}$ is the associated Cartan $1$-form. Since the notion of semi-conformal is conformally invariant and since any Riemannian surface is locally conformally equivalent to a domain of $\RR^2$ with its standard metric, for the rest of this section, we suppose the frame $\{ f_1, f_2\}$ parallel, so the connection form $\rho_{12}$ vanishes. By a trick, we will later remove this assumption in our expression for the Ricci curvature. 

 Let $\{ e_1, e_2,e_3, e_4\}$ be a positive orthonormal frame on $M^4$ such that $\dd \phi (e_i) = \la f_i$ for $i = 1,2$, and $e_3, e_4 \in V := \ker \dd \phi$.  We will use indices in the following way: $i,j, \ldots \in \{ 1,2\}$, $r,s, \ldots \in \{ 3,4\}$, $a,b, \ldots  \in \{ 1,2,3,4\}$ and sum over repeated indices. At each $x \in M$, let $\Hh_x : T_xM \ra H_x = V_x^{\perp}$ denote orthogonal projection onto the horizontal space. If we don't wish to be specific about the point $x$ we will simply write $\Hh$. Similarly, $\Vv$ denotes projection onto the vertical space.    
 
 Define complementary indices $i', j', \ldots$ by $i' = 2$ if $i=1$ and $i'=1$ if $i=2$.  Set $J^H$ to be rotation by $+\pi /2$ in the horizontal space $H$, thus: $J^H(e_1)=e_2$ and $J^H(e_2)=-e_1$, equivalently $J^H(e_i) = (-1)^{i+1}e_{i'}$.  Similarly, set $J^V$ to be rotation by $+\pi /2$ in the vertical space $V$, thus: $J^V(e_3) = e_4$ and $J^V(e_4) = -e_3$.  Then $J:= (J^H, J^V)$ defines an almost Hermitian structure on $(M,g)$.  
 
 \begin{definition} \label{def:int-form} For a semi-conformal submersion as above, define the \emph{integrability $1$-form} $\zeta : TM \ra \RR$ by 
$$
\zeta (X) := g(\na_{e_1}e_2, \Vv (X)) = \tfrac{1}{2}g([e_1,e_2], \Vv (X)) \qquad \forall \ X \in TM \,,
$$
where $\Vv$ is orthogonal projection onto $\ker \dd \phi$ and the second equality follows from \emph{Lemma \ref{lem:one}(i)} below. Then, $\zeta$ is well-defined independently of the (positive) horizontal orthonormal frame $\{ e_1, e_2\}$ and vanishes if and only if the horizontal distribution is integrable.  
\end{definition} 

\begin{definition} \label{def:2nd-ff} Let $\Ss = \phi^{-1}(y)$ be a fibre of $\phi$. Then for vector fields $X,Y$ tangent to $\Ss$, we have
$$
\na_XY = \na^{\Ss}_XY + B_XY
$$
where $\na$ is the connection on $M$, $\na^{\Ss}$ the connection on $\Ss$, i.e. $\na^{\Ss}_XY = \Vv \na_XY$, and $B$ is the \emph{second fundamental form of $\Ss$} (symmetric by integrability of the vertical distribution).  Then the \emph{mean curvature of the fibre} $\mu := \tfrac{1}{2}\Tr B = \tfrac{1}{2}\Hh (\na_{e_3}e_3 + \na_{e_4}e_4)$.

Extend $B$ to all vectors by the formula $B_XY:= \Hh \na_{\Vv X}\Vv Y$. Then its adjoint is characterized by:
$$
g(B_XY,Z) = g(Y, B^*_XZ) \quad \Rightarrow \quad B^*_XZ = - \Vv \na_{\Vv X}\Hh Z\,.
$$
\end{definition}

\begin{lemma} \label{lem:fund-eq} {\rm (Fundamental equation of a semi-conformal submersion \cite{Ba-Wo})}
For a semi-conformal submersion $\phi : (M^m, g) \ra (N^n,h)$, the tension field $\tau_{\phi} = \Tr_g\na \dd \phi$ is given by 
$$
\tau_{\phi} = - (n-2) \dd \phi (\grad \ln \la ) - (m-n) \dd \phi (\mu )
$$
where $\mu$ is the mean-curvature of the fibres. 
\end{lemma}

Recall that the connection forms $\om_{ab}$ are defined by $\na e_a = \sum_b\om_{ab}e_b$.  In order to express the connection coefficients, we require only the form $\om_{34}$.  The following lemma expresses the connection coefficients in terms of the above quantities.  

\begin{lemma} \label{lem:one}
$$
\begin{array}{llcl}
{\rm (i)} & \na_{e_i}e_j & = & - e_j(\ln \la )e_i + \grad \ln \la + (-1)^{i+1}\delta_{ij'} \zeta^{\sharp} \\
{\rm (ii)} & \na_{e_i}e_r & = & - e_r(\ln \la )e_i - \zeta (e_r)Je_i + \om_{34}(e_i)Je_r \\
{\rm (iii)} & \na_{e_r}e_i & = & - \zeta (e_r)Je_i - B^*_{e_r}e_i \\
{\rm (iv)} & \na_{e_r}e_s & = & B_{e_r}e_s + \om_{34}(e_r)Je_s \,.
\end{array}
$$
\end{lemma}

\begin{proof} (i) From Lemma \ref{lem:fund-eq}, 
$$
\tau_{\phi} = - 2\dd \phi (\mu )\,.
$$
But, recalling we sum over repeated indices, $\na \dd \phi (e_r,e_r) = - \dd \phi (\na_{e_r}e_r) = - 2\dd \phi (\mu )$, so that
$$
 \na \dd \phi (e_i, e_i) = \tau_{\phi} - \na \dd \phi (e_r,e_r) = 0\,.
$$
On the other hand,
\begin{eqnarray*}
 & &  \na \dd \phi (e_i, e_i) 
 =   \left( - \dd \phi (\na_{e_i}e_i) + \na_{e_i}^{\phi^{-1}}\dd \phi (e_i)\right) \\
 & = &   \left( - \dd \phi (\na_{e_i}e_i) + e_i(\ln \la ) \dd \phi (e_i) + \la^2\na^N_{f_i}f_i\right) \\
 & = &  \left( - \dd \phi (\na_{e_i}e_i) + e_i(\ln \la ) \dd \phi (e_i) \right)  
  =   \left( - \dd \phi (\na_{e_i}e_i) + e_i(\ln \la ) \dd \phi (e_i) \right) \,.
\end{eqnarray*}
The expression for the horizontal component of $\na_{e_i}e_j$ now follows when we note that $g(e_1, \na e_1) = 0$ etc.  

For the vertical component, first note that 
\begin{equation} \label{bracket-ir}
g([e_{r}, e_i], e_j ) = e_{r}(\ln \la ) g(e_i, e_j)  \qquad (\forall \ i,j, \in \{ 1,2\} \quad \forall \ r \in \{ 3, 4\})\,,
\end{equation} 
since, on the one hand
$$
\na \dd \phi (e_i, e_r) = - \dd \phi (\na_{e_i}e_r)\,;
$$
on the other hand, by the symmetry of the second fundamental form
\begin{eqnarray*}
\na \dd \phi (e_i, e_r) = \na \dd \phi (e_r, e_i) & = & - \dd \phi (\na_{e_r}e_i) + \na_{e_r}^{\phi^{-1}}\dd \phi (e_i) \\
& = & - \dd \phi (\na_{e_r}e_i) + e_r(\ln \la )\dd \phi (e_i) \\
& \Longrightarrow & \dd \phi (\na_{e_i}e_r) = \dd \phi (\na_{e_r}e_i)- e_r(\ln \la )\dd \phi (e_i)\,.
\end{eqnarray*}
Equation \eqref{bracket-ir} follows. But then
\begin{eqnarray*}
- g(\na_{e_i}e_j, e_r) & = & g(e_j, \na_{e_i}e_r) = g(e_j, \na_{e_r}e_i) - e_r(\ln \la )g(e_j, e_i) \\
- g(\na_{e_j}e_i, e_r) & = & g(e_i, \na_{e_j}e_r) = g(e_i, \na_{e_r}e_j) - e_r(\ln \la )g(e_i, e_j)
\end{eqnarray*}
Now add and use the fact that $0 = e_r(g(e_i, e_j)) = g(\na_{e_r}e_i, e_j) + g(e_i, \na_{e_r}e_j)$.  

\noindent (ii) follows since 
\begin{eqnarray*}
\Hh \na_{e_i}e_r & = & g(\na_{e_i}e_r, e_j)e_j = - g(e_r, \na_{e_i}e_j)e_j \\
& = & - e_r(\ln \la )e_i + (-1)^i\zeta(e_r)e_{i'} = - e_r(\ln \la )e_i - \zeta(e_r) J^{\Hh}e_i 
\end{eqnarray*}

\noindent (iii) follows from \eqref{bracket-ir} and (ii). 

\noindent (iv) is a consquence of the definitions. 
\end{proof} 

\begin{corollary} \label{cor:one} 
$$
\begin{array}{llcl}
{\rm (i)}  & 
[e_i,e_j] & = & e_i(\ln \la ) e_j - e_j(\ln \la )e_i + 2(-1)^{i+1}\delta_{ij'}\zeta^{\sharp} \\
{\rm (ii)} & [e_r, e_i] & = & e_r(\ln \la )e_i - B^*_{e_r}e_i - \om_{34}(e_i)Je_r \\
{\rm (iii)} & \na_{e_i}e_i & = & \grad \ln \la + \Vv \grad \ln \la \\
{\rm (iv)} & \na_{e_a}e_a & = & \grad \ln \la + \Vv \grad \ln \la + 2 \mu + \om_{34}(e_r)Je_r\,.
\end{array} 
$$
\end{corollary}

\section{The Ricci curvature}  \label{sec:ricci}
\n Let $\phi : (M^4,g) \ra (N^2,h)$ be a semi-conformal submersion between oriented Riemannian manifolds.  Choose an orthonormal frame field $\{ e_a\} = \{ e_i ; e_r\}$ adapted to the horizontal and vertical spaces.  
The Ricci curvature is determined by its components:
$$
\Ric = R_{ab}\ta_a\ta_b = R_{11}\ta_1{}^2 + 2R_{12}\ta_1\ta_2 + \cdots
$$
where $\{ \ta_a\}$ is the dual frame to $\{ e_a\}$ and the product $\ta_a\ta_b = \ta_a \odot \ta_b = \tfrac{1}{2}(\ta_a\otimes\ta_b + \ta_b\otimes \ta_a)$ is the symmetric product of $1$-forms.  The coefficients $R_{ab}$ are symmetric in their indices and $R_{ab} = \Ric (e_a, e_b)$. In order to compute the Ricci curvature associated to a semi-conformal submersion, we will separately calculate the horizontal components $R_{ij}$, the mixed components $R_{ri}$ and the vertical components $R_{rs}$.  

 Define the covariant tensor fields $C$ and $C^*$ by 
\begin{eqnarray*}
C(X,Y) & :=  & g(B_{e_r} X, B_{e_r}Y) = g(\Tr (B^*B)(X), Y)  \\
 C^*(X,Y) & := & g(B^*_{e_r}X, B^*_{e_r}Y) = g(\Tr(BB^*)(X), Y) 
\end{eqnarray*}
Note that $C$ and $C^*$ are well-defined independent of the frame, symmetric and that $C$ vanishes on horizontal vectors and $C^*$ on vertical vectors.  

For a general covariant tensor field $T (X,Y,Z, \ldots )$, define its divergence as deriviation and contraction with respect to the \emph{first} entry:
\begin{eqnarray*}
(\dv T)(Y,Z, \ldots ) & = & (\na_{e_a}T)(e_a, Y, Z, \ldots ) = e_a(T(e_a, Y, Z, \ldots )) - T(\na_{e_a}e_a, Y, Z, \ldots) \\
 & & \quad  - T(e_a, \na_{e_a}Y, Z, \ldots ) - T(e_a, Y, \na_{e_a}Z, \ldots ) - \cdots
\end{eqnarray*}

To the second fundamental form of the fibres $B$ (a $(2,1)$ tensor field), we associate two $(3,0)$-tensor fields. The first of these is $B_1 : TM \times TM \times TM \to \RR$ determined by
$$
B_1(X,Y,Z) = g(X, \Hh \na_{\Vv Y} \Vv Z)
$$
and the second
$$
B_2(X,Y,Z) = g( \Hh \na_{\Vv X}\Vv Y, Z)
$$
Note that $B_1$ and $B_2$ are identical up to ordering of their arguments, however, their divergences differ.

Our aim is to calculate the Ricci curvature in terms of parameters associated to $\phi$. Being a tensorial object, it suffices to calculate $\Ric$ at a point $x_0$ where we can suppose the frame chosen such that $\Vv \na_{e_r}e_s = 0$, for all $r,s = 3,4$. Such a frame can be constructed by first choosing a local \emph{normal} frame $\{ e_r\}$ for the fibre $\phi^{-1}(\phi (x_0))$ centred on $x_0$ (see \cite{Sp}, Vol. 2, Chapter 7) and then extending this to an orthonormal frame $\{ e_a\}$ about $x_0$ in $M$. In particular, at $x_0$, we have $\om_{34}(e_r) = 0$ for $r = 3,4$.

\begin{lemma} \label{lem:div-B}  Acting on vertical vectors, the divergence of $B_1$ at $x_0$ is determined by, 
\begin{eqnarray*}
(\dv B_1)(e_r,e_s) & = &  e_i(g(e_i,B_{e_r}e_s)) - 2\mu^{\flat}(B_{e_r}e_s) - \dd \ln \la (B_{e_r}e_s) \\
 & & \qquad - g(e_t, \na_{e_i}e_r)g(e_i, B_{e_t}e_s) - g(e_t, \na_{e_i}e_s)g(e_i,B_{e_r}e_t) 
\end{eqnarray*}
(recalling, we sum over repeated indices).
\end{lemma}

\begin{proof}
\begin{eqnarray*}
(\dv B_1)(e_r,e_s) & = &  (\na_{e_a}B_1)(e_a,e_r,e_s) \\
 & = & e_i(B_1(e_i,e_r,e_s)) - B_1(\na_{e_a}e_a,e_r,e_s) - B_1(e_i,\na_{e_i}e_r,e_s) - B_1(e_i,e_r,\na_{e_i}e_s)
\end{eqnarray*}
From Corollary \ref{cor:one}(iv), at $x_0$, $\Hh \na_{e_a}e_a = 2\mu + \Hh \grad \ln \la$; also $\Vv \na_{e_i}e_r = g(e_t, \na_{e_i}e_r)e_r$ etc. and the formula follows. 
\end{proof}

\begin{lemma} \label{lem:div-B-2} Acting on a vertical and a horizontal vector, the divergence of $B_2$ at $x_0$ is given by
$$
(\dv B_2)(e_r,e_i) = e_s(B_2(e_s, e_r, e_i)) - 2g(B_{\Vv \grad \ln \la}e_r, e_i) - \zeta (\na_{e_r}J^{\Hh}e_i) 
$$
\end{lemma}

\begin{proof} Calculating at $x_0$,
\begin{eqnarray*}
(\dv B_2)(e_r,e_i) & = &   e_a(B_2(e_a, e_r, e_i)) - B_2(\Vv \na_{e_a}e_a, e_r, e_i) - B_2(e_s,\Vv\na_{e_s}e_r, e_i) - B_2(e_s, e_r, \Hh \na_{e_s}e_i)   \\
 & = & e_s(B_2(e_s,e_r,e_i)) - B_2(\Vv \na_{e_j}e_j, e_r, e_i) - B_2(e_s, e_r, \Hh \na_{e_s}e_i)
 \end{eqnarray*}
On applying Corollary \ref{cor:one}(iii) and Lemma \ref{lem:one}(iii), this becomes
$$
e_s(B_2(e_s, e_r, e_i)) - 2 g(B_{\Vv\grad\ln \la} e_r, e_i) + \zeta (e_s)g(\na_{e_r}e_s, J^{\Hh}e_i)
$$
But the latter term equals $- \zeta (e_s)g(e_s, \na_{e_r}J^{\Hh}e_i)$ and the formula follows. 

\end{proof}

In what follows, we shall first establish the stated formulae for the case when $N^2$ is flat; in particular, we can suppose that $\dd \phi (e_i) = \la f_i$ where $\{ f_i\}$ is a parallel frame: $\na f_i = 0$ and apply the formulae of \S \ref{sec:first}. We will then extend the formulae to the case when $N^2$ is an arbitrary Riemannian surface.

\subsection{Horizontal components of the Ricci curvature}\label{sec:hor-ricci} 

\n First, we require the following lemma.

\begin{lemma} \label{lem:hor-sec}  The horizontal sectional curvature $K^{H}:= g(R(e_1,e_2)e_2, e_1)$ is given by 
$$
K^{H} = \Delta \ln \la - \Tr_{\Vv}\na \dd \ln \la + ||\Vv\grad \ln \la ||^2 - 3||\zeta ||^2
$$
\end{lemma}
\begin{proof} From Lemma \ref{lem:one}(i) and Corollary \ref{cor:one}(i),  
\begin{eqnarray*}
K^{H} & = &  g(\na_{e_1}\na_{e_2}e_2 - \na_{e_2}\na_{e_1}e_2 - \na_{[e_1,e_2]}e_2, e_1) \\
& = & g(\na_{e_1}(e_1(\ln \la ) e_1 + \Vv \grad \ln \la ) + \na_{e_2}(e_2(\ln \la )e_1 - \zeta^{\sharp}), e_1) \\
 & & \quad  - e_1(\ln \la )g(\na_{e_2}e_2, e_1) + e_2(\ln \la )g(\na_{e_1}e_2, e_1) - 2 g(\na_{\zeta^{\sharp}}e_2,e_1) \\
  & = & e_1(e_1(\ln \la )) + e_2(e_2(\ln \la )) - ||\Vv \grad \ln \la ||^2 - ||\zeta ||^2 \\
   & & \quad - e_1(\ln \la )^2 - e_2(\ln \la )^2 - 2 \zeta (e_r)g(\na_{e_r}e_2 , e_1) \\
    & = & \Delta (\ln \la ) - \Tr_{\Vv}\na \dd \ln \la + \dd \ln \la (\na_{e_i}e_i) - ||\Vv \grad \ln \la ||^2 - ||\Hh \grad \ln \la ||^2 - 3 ||\zeta ||^2,
\end{eqnarray*}
which, from Corollary \ref{cor:one}(iii), gives the required formula. 
\end{proof}
\begin{lemma} \label{lem:ricci-hor}  The horizonal part of the Ricci curvature: $\Ric\vert_{H\times H}$ is given by 
$$
\Ric\vert_{H\times H}  =  \left\{ \la^2K^N + \Delta \ln \la + 2\dd \ln \la (\mu ) - 2 ||\zeta ||^2 \right\} g^{H} 
-C^* + \Ll_{\mu}g \vert_{H \times H}
$$
where $K^N$ denotes the Gaussian curvature of $N$. 
\end{lemma}

\begin{proof} The horizontal components $R_{ij} = \Ric (e_i, e_j)$ are given by 
$$
R_{ij} = g(R(e_i, e_a)e_a, e_j) = K^{H} g(e_i, e_j) + g(R(e_i, e_r)e_r, e_j)\,. 
$$
where $K^H$ is given by Lemma \ref{lem:hor-sec} above.  

We now calculate $g(R(e_i, e_r)e_r, e_j) = g(\na_{e_i}\na_{e_r}e_r - \na_{e_r}\na_{e_i}e_r - \na_{[e_i,e_r]}e_r, e_j)$. 
Then
\begin{eqnarray*}
g(\na_{e_i}\na_{e_r}e_r, e_j) & = & g(\na_{e_i}(\Hh \na_{e_r}e_r + \Vv \na_{e_r}e_r), e_j) \\
 & = & 2g(\na_{e_i}\mu , e_j) - g(\Vv \na_{e_r}e_r, \na_{e_i}e_j) = 2g(\na_{e_i}\mu , e_j) \,.
\end{eqnarray*}
From Lemma \ref{lem:one}(ii) and (iii), 
\begin{eqnarray*}
- g(\na_{e_r}\na_{e_i}e_r, e_j) & = & - g(\na_{e_r}(\Hh\na_{e_i}e_r + \Vv\na_{e_i}e_r), e_j) \\
 & = & g(\na_{e_r}(e_r(\ln \la )e_i + \zeta (e_r)Je_i), e_j) - g(\Vv \na_{e_i}e_r, \na_{e_r}e_j) \\
 & = & e_r(e_r(\ln \la ))g(e_i,e_j) + e_r(\ln \la )g(\na_{e_r}e_i, e_j) + g(\na_{e_r}(\zeta (e_r)Je_i), e_j) \\
  & & \qquad  - g(\Vv \na_{e_i}e_r, \na_{e_r}e_j) \\
   & = & (\Tr_V \na \dd \ln \la + 2\dd \ln \la (\mu))g(e_i,e_j) \\
    & & \quad + e_r(\ln \la )g(\na_{e_r}e_i, e_j) + g(\na_{e_r}(\zeta (e_r)Je_i), e_j)- g(\Vv \na_{e_i}e_r, \na_{e_r}e_j)\,. 
\end{eqnarray*}
From Lemma \ref{lem:one},
\begin{eqnarray*}
[e_i, e_r] & = & g([e_i,e_r], e_k)e_k + g([e_i,e_r], e_s)e_s \\
 & = & - e_r(\ln \la )e_i + g(\na_{e_i}e_r - \na_{e_r}e_i, e_s)e_s
\end{eqnarray*}
so that
\begin{eqnarray*}
- g(\na_{[e_i,e_r]}e_r, e_j) & = & e_r(\ln \la ) g(\na_{e_i}e_r, e_j) -  g(\na_{e_i}e_r - \na_{e_r}e_i, e_s)g(\na_{e_s}e_r, e_j) \\
 & = & e_r(\ln \la ) g(- e_r(\ln \la ) e_i - \zeta (e_r) Je_i, e_j) \\
  & & 
  \qquad  - g(\na_{e_i}e_r, e_s)g(\na_{e_s}e_r, e_j) + g(\na_{e_r}e_i, e_s)g(\na_{e_s}e_r, e_j) \\
   &  = & - ||\Vv \grad \ln \la ||^2g(e_i, e_j) - e_r(\ln \la )\zeta (e_r)g(Je_i, e_j) \\ 
   & & \qquad - g(\na_{e_i}e_r, e_s)g(\na_{e_s}e_r, e_j) + g(\na_{e_r}e_i, e_s)g(\na_{e_s}e_r, e_j) \,.
\end{eqnarray*}
However, the Ricci tensor is symmetric in its arguments: $\Ric (e_i, e_j) = \tfrac{1}{2}(\Ric (e_i, e_j) + \Ric (e_j, e_i)$.  But then $g(\na_{e_i}\mu , e_j) + g(\na_{e_j}\mu , e_i) = \Ll_{\mu}g (e_i, e_j)$\,, $g(Je_i, e_j) + g(Je_j, e_i) = 0$ and 
$$
\zeta (e_r) (g(\na_{e_r}Je_i, e_j) + g(\na_{e_r}Je_j, e_i)= - \zeta (e_r)(g(Je_i, \na_{e_r}e_j) + g(Je_i, \na_{e_r}e_j)=   ||\zeta ||^2\,.
$$
Collecting terms now give the required expression in the case of flat codomain. 
\end{proof}

\subsection{The mixed components of the Ricci curvature} \label{sec:mixed-ricci} 

\begin{lemma} \label{lem:ric-mixed}
For $X$ a horizontal vector and $U$ a vertical vector, one has
\begin{eqnarray*}
\Ric (X,U) & = & \na \dd \ln \la (X,U) - (\dd \ln \la )^2(X,U) - 2 (\dd \ln \la \odot \zeta)(JX, U) - (\na_{JX}\zeta )(U) \\
& & \quad - 2 \zeta (\na_U JX)  - \dv B_2(U,X) - 2\dd\ln \la (B^*_UX) + 2 (\na_U\mu^{\flat})(X)\,.
\end{eqnarray*}
\end{lemma}

\begin{proof} 
By tensoriality, it suffices to set $X = e_i$ and $U = e_r$.  Then
$$
\Ric (e_i, e_r) = g(R(e_i, e_a)e_a, e_r) = g(R(e_i,e_j)e_j,e_r) + g(R(e_r,e_s)e_s,e_i)\,.
$$
First, we deal term by term with 
$$
g(R(e_i,e_j)e_j,e_r) = g(\na_{e_i}\na_{e_j}e_j - \na_{e_j}\na_{e_i}e_j-\na_{[e_i,e_j]}e_j, e_r)\,.
$$
From Corollary \ref{cor:one}(iii) and Lemma \ref{lem:one}(ii),
\begin{eqnarray*}
g(\na_{e_i}\na_{e_j}e_j, e_r) & = & g(\na_{e_i}(2\grad \ln \la - \Hh \grad \ln \la ), e_r) \\
 & = & 2\na \dd \ln \la (e_i, e_r) + g(\Hh \grad \ln \la , \na_{e_i}e_r ) \\ 
 & = & 2\na \dd \ln \la (e_i, e_r) - e_i(\ln \la )e_r(\ln \la ) - \zeta (e_r)(Je_i)(\ln \la )\,.
\end{eqnarray*}
Also, from Lemma \ref{lem:one}(ii),
\begin{eqnarray*}
- g(\na_{e_j}\na_{e_i}e_j, e_r) & = & - e_j(g(\na_{e_i}e_j, e_r) + g(\na_{e_i}e_j, \na_{e_j}e_r) = e_j(g(e_j, \na_{e_i}e_r)) + g(\na_{e_i}e_j, \na_{e_j}e_r)  \\
 & = & e_j\big(-e_r(\ln \la ) \delta_{ij} - \zeta (e_r)g(e_j, Je_i)\big) + g(\na_{e_i}e_j, \na_{e_j}e_r) \\
 & = & - e_i(e_r(\ln \la )) - (Je_i)(\zeta (e_r)) + g(\na_{e_i}e_j, \na_{e_j}e_r ) \\ 
 & = & - e_i(e_r(\ln \la )) - (\na_{Je_i}\zeta ) (e_r) - \zeta (\na_{Je_i}e_r) + g(\na_{e_i}e_j, \na_{e_j}e_r)\,, 
\end{eqnarray*}
where, from Lemma \ref{lem:one}, 
\begin{eqnarray*}
g(\na_{e_i}e_j, \na_{e_j}e_r) & = & g(\na_{e_i}e_j, e_k)g(e_k, \na_{e_j}e_r) + g(\na_{e_i}e_j, e_s)g(e_s, \na_{e_j}e_r) \\
 & = & e_k(\ln \la )\delta_{ij}g(e_k, \na_{e_j}e_r) - e_j(\ln \la )\delta_{ik}g(e_k, \na_{e_j}e_r)  \\
  & & \qquad + e_s(\ln \la )\delta_{ij}g(e_s, \na_{e_j}e_r) + (-1)^{i+1}\delta_{ij'}\zeta (e_s)g(e_s, \na_{e_j}e_r) \\
  & = & g(\Hh \grad \ln \la , \na_{e_i}e_r) - g(e_i, \na_{e_j}e_r)e_j(\ln \la ) \\
   & & \qquad  + g(\Vv \grad \ln \la , \na_{e_i}e_r) + \zeta (e_s)g(e_s, \na_{Je_i}e_r) \\
   & = & - 2\zeta (e_r)\dd \ln \la (Je_i) + \dd \ln \la (\Vv \na_{e_i}e_r) + \zeta (\na_{Je_i}e_r)\,.
\end{eqnarray*}
From Corollary \ref{cor:one}(i) and Lemma \ref{lem:one}(ii), 
\begin{eqnarray*}
- g(\na_{[e_i, e_j]}e_j, e_r) & = & - e_i(\ln \la )\delta_{jk}g(\na_{e_k}e_j, e_r) + e_j(\ln \la )\delta_{ik}g(\na_{e_k}e_j, e_r) + 2(-1)^i\delta_{ij'}\zeta (e_s)g(\na_{e_s}e_j, e_r) \\
& = & - 2e_i(\ln \la )e_r(\ln \la ) - g(\Hh \grad \ln \la , \na_{e_i}e_r) - 2 \zeta (e_s)g(\na_{e_s}Je_i, e_r) \\ 
& = & -  e_i(\ln \la ) e_r (\ln \la ) + \zeta (e_r)\dd \ln \la (Je_i) - 2 \zeta (\na_{e_r}Je_i) \,.
\end{eqnarray*}
Collecting terms now yields
$$
g(R(e_i, e_j)e_j, e_r) = \na \dd \la (e_i, e_r) - e_i(\ln \la )e_r(\ln \la ) - \zeta (e_r)\dd \ln \la (Je_i) - (\na_{Je_i}\zeta )(e_r) - 2 \zeta (\na_{e_r}Je_i) \,.
$$
For the other term, first note that at the point $x_0$, 
$$
g(\na_{e_r}\na_{e_s}e_s, e_i) = g(\na_{e_r}(\Hh \na_{e_s}e_s + \Vv \na_{e_s}e_s), e_i) = 2g(\na_{e_r}\mu , e_i) - g(\Vv \na_{e_s}e_s, \na_{e_r}e_i) = 2g(\na_{e_r} \mu , e_i)\,.
$$
 Then from Lemma \ref{lem:div-B-2}, 
\begin{eqnarray*}
g(R(e_r, e_s)e_s, e_i) & = & g(\na_{e_r}\na_{e_s}e_s - \na_{e_s}\na_{e_r}e_s - \na_{[e_r, e_s]}e_s, e_i) \\
 & = & 2g(\na_{e_r}\mu , e_i) - e_s(g(\na_{e_r} e_s, e_i)) + g(\na_{e_r}e_s, \na_{e_s}e_i ) - g(\na_{[e_r, e_s]}e_s, e_i) \\
  & = & 2(\na_{e_r}\mu^{\flat})(e_i) - (\dv B_2)(e_r, e_i) - 2 g(\na_{e_r}\Vv \grad \ln \la , e_i) \\
   & & \qquad - \zeta (\na_{e_r}Je_i) + g(\Hh \na_{e_r}e_s, \Hh \na_{e_r}e_i)\,. 
\end{eqnarray*} 
But from from Lemma \ref{lem:one},  $g(\Hh \na_{e_r}e_s, \Hh \na_{e_r}e_i) = - g(\na_{e_r}e_s, \zeta (e_s)Je_i) = \zeta (\na_{e_r}Je_i)$.  The formula now follows for flat codomain. 
\end{proof}

\subsection{The vertical components of the Ricci curvature} \label{sec:vert-ricci} 
Define the vertical sectional curvature by $K^V:= g(R^F(e_3, e_4)e_4, e_3)$ where $F = \phi^{-1}(y) \subset M$ is the fibre over $y \in N$ and $R^F$ is the Riemannian curvature of $F$.  Then $K^V$ is related to the sectional curvature in $M$ via the Gauss equation (see \cite{Sp} Chapter 7):
$$
g(R(e_3, e_4)e_4, e_3) =  g(R^F(e_3, e_4)e_4, e_3) + |B_{e_3}e_4|^2 - g(B_{e_3}e_3, B_{e_4}e_4)
$$
The correction terms have an invariant expression given by the following lemma, established by evaluating the right-hand and left-hand sides on the various $(e_r, e_s)$.
\begin{lemma} \label{lem:id-mu-C}
$$
(|B_{e_3}e_4|^2 - g(B_{e_3}e_3, B_{e_4}e_4))g^V = C - 2\mu^{\flat}(B_{\star}\star )
$$
\end{lemma}

\begin{lemma} \label{ric-vert}
$$
\Ric\vert_{V\times V}  = K^{V}g^{V} + 2 \na \dd \ln\la\vert_{V\times V} + 2\dd \ln \la (B_{\star}\star ) - 2 (\dd \ln \la )^2\vert_{V\times V}  + 2 \zeta^2 +\dv B_1\vert_{V\times V} 
$$
\end{lemma}
\begin{proof}
$$
\Ric (e_r, e_s) = g(R(e_r, e_a)e_a, e_s) = (K^V+|B_{e_3}e_4|^2 - g(B_{e_3}e_3, B_{e_4}e_4))g(e_r,e_s) + g(R(e_r,e_i)e_i, e_s)\,,
$$
with
$$
g(R(e_r, e_i)e_i, e_s) = g(\na_{e_r}\na_{e_i}e_i - \na_{e_i}\na_{e_r}e_i - \na_{[e_r, e_i]}e_i, e_s)\,.
$$
From Corollary \ref{cor:one}(iii), $\na_{e_i}e_i = \grad \ln \la + \Vv \grad \ln \la = 2 \grad \ln \la - \Hh \grad \ln \la$, so that
\begin{eqnarray*}
g(\na_{e_r}\na_{e_i}e_i, e_s)  & = & 2g(\na_{e_r}\grad \ln \la , e_s) + g(\Hh \grad \ln \la , \na_{e_r}e_s) \\
& = & 2 \na \dd \ln \la (e_r,e_s) + \dd \ln \la (B_{e_r}e_s)\,.
\end{eqnarray*}
From Lemma \ref{lem:div-B}, 
\begin{eqnarray*}
- g(\na_{e_i}\na_{e_r}e_i, e_s) & = & - e_i(g(\na_{e_r}e_i, e_s) + g(\na_{e_r}e_i, \na_{e_i}e_s) \\
& = & \dv B_1(e_r, e_s) + 2\mu^{\flat}(B_{e_r}e_s) + \dd \ln \la (B_{e_r}e_s) \\
 & & 
 \quad + g(\na_{e_r}e_i, e_j)g(e_j, \na_{e_i}e_s) + g(\na_{e_r}e_i, e_t)g(e_t, \na_{e_i}e_s) \\
 & = & \dv B_1(e_r, e_s) + 2\mu^{\flat}(B_{e_r}e_s) + \dd \ln \la (B_{e_r}e_s) \\
 & & 
 \quad + g(e_t, \na_{e_i}e_r)g(e_i, \na_{e_t}e_s) + g(\na_{e_r}e_i, e_j)g(e_j, \na_{e_i}e_s)\,,
 \end{eqnarray*}
 where the last term can be expressed using Lemma \ref{lem:one}(ii) and (iii): 
 $$
 g(\na_{e_r}e_i, e_j)g(e_j, \na_{e_i}e_s) = 2 \zeta (e_r)\zeta (e_s)\,.
 $$
From Corollary \ref{cor:one}(ii) and (iii) 
$$
- g(\na_{[e_r, e_i]}e_i, e_s) = - 2e_r(\ln \la )e_s(\ln \la ) - g(e_i, B_{e_r}e_t)g(e_i, B_{e_t}e_s) + g(e_t, \na_{e_i}e_r)g(\na_{e_t}e_i, e_s)\,.
$$
On collecting terms and applying Lemma \ref{lem:id-mu-C}, the formula follows for the case of flat codomain. 
\end{proof}

\subsection{Mapping into an arbitrary curved surface} \label{sec:curved-target}

\n Suppose $\phi : (M^4, g) \ra (N^2, h)$ is a semi-conformal submersion into an arbitrary Riemannian surface with dilation $\la$.  About a point in the image of $\phi$, choose local isothermal coordinates $\psi : W \ra \RR^2$ on an open set $W\subset N^2$, so that $h = \nu^{-2}(\dd y_1{}^2 + \dd y_2{}^2)$ for some function $\nu : W \ra \RR$.  
Consider the following composition:
$$
(M^4, g) \stackrel{\phi}{\longrightarrow} (W\subset N^2, h) \stackrel{\psi}{\longrightarrow} (W' \subset \RR^2, \ov{h}) 
$$
where $\ov{h}$ is the canonical metric $\dd y_1{}^2 + \dd y_2{}^2$ on $\RR^2$ and $W'=\psi (W)$.  Then the formulae of \S\ref{sec:hor-ricci}, \S\ref{sec:mixed-ricci} and \S\ref{sec:vert-ricci} apply to $\psi \circ \phi$. We now show how they extend to $\phi$.  

\begin{lemma} \label{lem:conf-cod}
$$
\la^2 K^N\circ \phi = \Delta \ln (\nu\circ \phi ) + 2 \dd \ln (\nu \circ \phi )(\mu )\,.
$$
\end{lemma}
\begin{proof} 
First note that $K^N = \nu^{-2}\Delta_{\ov{h}}\ln \nu = \Delta_h\ln \nu$.  Then from Lemma \ref{lem:fund-eq}, 
\begin{eqnarray*}
\Delta_g (\ln \nu \circ \phi ) & = & \dd \ln \nu (\tau_{\phi}) + \Tr_g\na \dd \ln \nu (\dd \phi , \dd \phi ) \\
& = & - 2 \dd (\ln \nu \circ \phi )(\mu ) + \la^2 (\Delta_h \ln \tau ) \circ \phi \\
& = & - 2 \dd (\ln \nu \circ \phi )(\mu ) + \la^2 K^N \circ \phi\,.
\end{eqnarray*}
\end{proof} 

Since the dilation of $\psi \circ \phi$ is given by $\la \nu$, from Lemma \ref{lem:ricci-hor} (for the flat case), 
$$
\Ric\vert_{H \times H} = \left\{ \Delta \ln (\la \nu ) + 2 \dd \ln (\la \nu )(\mu ) - 2 ||\zeta ||^2 \right\} g^{\Hh} - C^* + \Ll_{\mu}g\vert_{H \times H}\,.
$$
But from Lemma \ref{lem:conf-cod},
$$
\Delta \ln (\la \nu ) + 2 \dd \ln (\la \nu )(\mu ) = \la^2 K^N + \Delta \ln \la + 2 \dd \ln \la (\mu )\,,
$$
where the latter quantity is invariant with respect to conformal changes of metric on the codomain.  

 For the mixed components of the Ricci curvature, we note that on setting $\ov{\la} = \la \nu$, 
\begin{eqnarray*}
\na \dd \ln \la (X,U) - (\dd \ln \la )^2(X,U) - 2 (\dd \ln \la \odot \zeta)(JX, U) \\
= \na \dd \ln \ov{\la} (X,U) - (\dd \ln \ov{\la} )^2(X,U) - 2 (\dd \ln \ov{\la} \odot \zeta)(JX, U).
\end{eqnarray*}
For example
$$
\na \dd \ln \ov{\la}(e_1, e_3) = \na \dd \ln \la (e_1, e_3) - \dd \ln (\nu \circ \phi )(\na_{e_1}e_3)
$$
But from Lemma \ref{lem:one}, 
\begin{eqnarray*}
- \dd \ln (\nu \circ \phi )(\na_{e_1}e_3) & = & - \dd \ln (\nu \circ \phi )(\Hh \na_{e_1}e_3) = \dd \ln (\nu \circ \phi )(g(e_3, \na_{e_1}e_1)e_1+g(e_3, \na_{e_1}e_2)e_2) \\
 & = &  2(\dd \ln \la \odot \dd \ln (\nu \circ \phi ))(e_1,e_3) + 2 (\dd \ln (\nu \circ \phi ) \odot \zeta)(Je_1, e_3)
\end{eqnarray*}
Whereas
\begin{eqnarray*}
- (\dd \ln \ov{\la} )^2(e_1,e_3) - 2 (\dd \ln \ov{\la} \odot \zeta)(Je_1, e_3)  =   - (\dd \ln \la )^2(e_1,e_3) - 2 (\dd \ln \la \odot \zeta)(Je_1, e_3) \\
 -2(\dd \ln \la \odot \dd \ln (\al \circ \phi ))(e_1,e_3) - 2 (\dd \ln (\al \circ \phi ) \odot \zeta)(Je_1, e_3)\,.
\end{eqnarray*}
The invariance of the vertical components of the Ricci curvature follows from the invariance of the quantity $\na \dd \ln\la\vert_{\Vv\times \Vv} + \dd \ln \la (B_{\star}\star )$, specifically $\na \dd \ln\la (e_r,e_s) + \dd \ln \la (B_{e_r}e_s ) = e_r(e_s(\ln \la ))- \dd \ln \la (\Vv \na_{e_r}e_s)= e_r(e_s(\ln \ov{\la} ))- \dd \ln \ov{\la} (\Vv \na_{e_r}e_s)$.

\section{Biconformal deformations} \label{sec:bic-def}

\subsection{The effect of a biconformal deformation on the Ricci curvature} \label{sec:bic}
Let $\phi : (M^4, g_0) \ra (N^2, h)$ be a semi-conformal map between oriented manifolds.  Consider a biconformal deformation:
$$
g = \frac{g_0^{H}}{\si^2} + \frac{g_0^{V}}{\rho^2} 
$$
where $\si , \rho : M^4 \ra \RR$ are smooth strictly positive functions.
Write objects with respect to $g_0$ with an index $0$, either upstairs or downstairs, and objects with respect to $g$ as before.  For example, the positive orthonormal basis with respect to $g_0$ will be written $\{ e_1^0, e_2^0, e_3^0, e_4^0\}$ and the dilation of $\phi$ with respect to $g_0$ as $\la_0$, etc.  Then the new frame field and the dual field of $1$-forms are given by
$$
e_1 = \si e^0_1,\ e_2 = \si e^0_2, \ e_3 = \rho e^0_3, \ e_4 = \rho e^0_4 \ \ta_1 = \frac{1}{\si} \ta^0_1,\ \ta_2 = \frac{1}{\si}\ta^0_2, \ \ta_3 = \frac{1}{\rho}\ta^0_3, \ \ta_4 = \frac{1}{\rho}\ta^0_4\,.
$$
The following lemma gives the change in the connection coefficients. 

\begin{lemma} \label{lem:cov-change}
$$
\begin{array}{crcl}
{\rm (i)} & g(\na_{e_r}e_s, e_i) & = & g_0(\na^0_{e_r^0}e_s^0, e_i) +  e_i(\ln\rho ) \delta_{rs}\\
{\rm (ii)} & g(\na_{e_i}e_r, e_s) & = & g_0(\na^0_{e_i}e_r^0,e_s^0)  \\
{\rm (iii)} & g(\na_{e_r}e_i,e_j) & = & g_0(\na^0_{e_r}e_i^0, e_j^0) + \frac{\rho^2-\si^2}{2\rho^2}g_0([e_i^0, e_j^0], e_r) \\
{\rm (iv)} & g(\na_{e_r}e_s, e_t) & = & g_0(\na^0_{e_r^0}e_s^0, e_t) + e_t(\ln \rho ) \delta_{rs} - e_s(\ln \rho ) \delta_{rt} \\
{\rm (v)} & g(e_i, \na_{e_k}e_j) & = & \si g_0(e_i^0, \na^0_{e_k^0}e_j^0) + \si (e_i^0(\ln \si ) \delta_{jk} - e_j^0(\ln \si ) \delta_{ik}) \\
{\rm (vi)} & g(\na_{e_i}e_j, e_r) & = & \frac{\si^2}{\rho^2}g_0(\na^0_{e_i^0}e_j^0, e_r) + \left( 1 - \frac{\si^2}{\rho^2}\right)e_r(\ln \la_0) \delta_{ij} + e_r(\ln \si ) \delta_{ij}
\end{array}
$$
\end{lemma}
\begin{proof} 
$$
 \begin{array}{ll}
 {\rm (i)} \quad 2 g(\na_{e_r}e_s, e_i)  &  =  g( [e_i,e_r], e_s) + g([e_i,e_s],e_r)\\
   & =  \si g_0([e_i^0,e_r^0],e_s^0) + \si g_0([e_i^0,e_s^0],e_r^0) + 2\si e_i^0(\ln\rho ) \delta_{rs} \\
  & =  2 \si g_0(\na^0_{e_r^0}e_s^0, e_i^0) + 2 \si e_i^0(\ln\rho ) \delta_{rs} \\
{\rm (ii)} \quad 2 g(\na_{e_i}e_r,e_s) & = g([e_i, e_r], e_s) - g([e_i,e_s],e_r) \\ 
 & =  \frac{1}{\rho^2}g_0([e_i, \rho e_r^0], \rho e_s^0) - \frac{1}{\rho^2} g_0([e_i, \rho e_s^0], \rho e_r^0) \\ 
 & =  g_0([e_i, e_r^0], e_s^0) - g_0([e_i, e_s^0], e_r^0) + e_i(\ln \rho ) \delta_{rs} - e_i(\ln \rho ) \delta_{rs} \\
 & = 2g_0(\na_{e_i}e_r^0, e_s^0) \\
{\rm (iii)} \quad 2g(\na_{e_r}e_i, e_j) & =  g([e_r, e_i], e_j) - g([e_i,e_j], e_r) + g([e_j, e_r], e_i) \\
 & =  \frac{1}{\si} g_0([\rho e_r^0, e_i^0], e_j^0) - \frac{1}{\rho} g_0([\si e_i^0, \si e_j^0], e_r^0) + \frac{1}{\si} g_0([\si e_j^0, \rho e_r^0], e_i^0) \\
& =  \rho g_0([e_r^0, e_i^0], e_j^0) + \frac{\rho}{\si}e_r^0(\si) \delta_{ij} - \frac{\si^2}{\rho}g_0([e_i^0, e_j^0], e_r^0) + \rho g_0([e_j^0, e_r^0], e_i^0) - \frac{\rho}{\si}e_r^0(\si) \delta_{ij} \\
 & =  2 \rho g_0(\na^0_{e_r^0}e_i^0, e_j^0) + \frac{\rho^2-\si^2}{\rho} g_0([e_i^0, e_j^0], e_r^0)
\end{array}  
$$ 

\n  (iv) As above, we write $2 g(\na_{e_r}e_s, e_t) = g([e_r,e_s], e_t) - g([e_s,e_t], e_r) + g([e_t,e_r], e_s)$ and replace $e_r$ by $\rho e_r^0$ etc. Case (v) is similar.
$$
 \begin{array}{ll}
{\rm (vi)} \quad 2 g(\na_{e_i}e_j, e_r) & =  g([e_i,e_j], e_r) - g([e_j,e_r], e_i) + g([e_r, e_i], e_j) \\
& =  \frac{1}{\rho}g_0([\si e_i^0,\si e_j^0], e_r^0) - \frac{1}{\si}g_0([\si e_j^0,\rho e_r^0], e_i^0) + \frac{1}{\si}g_0([\rho e_r^0, \si e_i^0], e_j^0) \\
& =  2\frac{\si^2}{\rho}g_0(\na^0_{e_i^0}e_j^0, e_r^0) + \frac{\si^2}{\rho} \left( g_0([e_j^0, e_r^0], e_i^0) - g_0([e_r^0, e_i^0], e_j^0)\right) \\
 &  \quad - \rho g_0([e_j^0, e_r^0], e_i^0) + \rho g_0([e_r^0, e_i^0], e_j^0) + \rho e_r^0(\ln \si ) \delta_{ij} + \rho e_r^0(\ln \si ) \delta_{ij} 
\end{array} 
$$
From Lemma \ref{lem:one}, this gives
$$
2 g(\na_{e_i}e_j, e_r) = 2\frac{\si^2}{\rho}g_0(\na^0_{e_i^0}e_j^0, e_r^0) - 2\frac{\si^2}{\rho} e_r^0(\ln \la_0)\delta_{ij} + 2\rho e_r^0(\ln \la_0) \delta_{ij} + 2\rho e_r^0(\ln \si ) \delta_{ij} 
$$
and the formula follows. 
\end{proof}

\begin{corollary} \label{cor:conn}
$$
\begin{array}{llll}
{\rm (i)} & \na_{e_s}e_j & = &  \si \rho \na^0_{e_s^0}e_j^0 + \frac{\rho^2-\si^2}{2\rho^2}\zeta^0(e_s)Je_j - e_j(\ln \rho )e_s \\
{\rm (ii)} & \na_{e_r}e_s & = & \si^2 \Hh \na^0_{e_r^0}e_s^0 + \rho^2 \Vv \na^0_{e_r^0}e_s^0 + \delta_{rs}\left( \si^2 \Hh \grad_{g_0}\ln \rho + \rho^2 \Vv \grad_{g_0}\ln \rho \right) - \rho^2 e_s^0(\ln \rho ) e_r^0 
\end{array}
$$
\end{corollary}
\begin{proof} From Lemma \ref{lem:cov-change}, 
\begin{eqnarray*}
\na_{e_s}e_j & = & g(\na_{e_s}e_j, e_i)e_i + g(\na_{e_s}e_j, e_r)e_r = g(\na_{e_s}e_j, e_i)e_i - g(e_j, \na_{e_s}e_r)e_r \\
 & = & g_0(\na^0_{e_s}e_j^0, e_i^0)e_i + \frac{\rho^2 - \si^2}{2\rho^2}g_0([e_j^0, e_i^0], e_s)e_i  - g_0(\na^0_{e_s^0}e_r^0, e_j)e_r - e_j(\ln \rho ) \delta_{rs}e_r  \\
 & = & \si \rho \na^0_{e_s^0}e_j^0 + \frac{\rho^2 - \si^2}{2\rho^2}\zeta^0(e_s)Je_j - e_j(\ln \rho ) e_s\,.
\end{eqnarray*}
 The proof of (ii) is similar.  
\end{proof}

\begin{corollary}\label{cor:B-change}
$$
B_{e_r}e_s = \si^2(B^0_{e_r^0}e_s^0 + g_0(e_r^0,e_s^0)\Hh \grad_{g_0} \ln \rho ) 
$$
\end{corollary}
\begin{proof}  From (i) of Lemma \ref{lem:cov-change},
$$
\ g_0(B_{e_r}e_s,e_i^0)  =  \si^2(g_0(\na^0_{e_r^0}e_s^0, e_i^0) + e_i^0(\ln\rho ) \delta_{rs})
$$
from which the formula follows. 
\end{proof}

\begin{lemma} \label{lem:bic-mu}  The mean curvature of the fibres, the integrability form and the dilation change according to
$$
\mu = \si^2(\mu_0 + \Hh \grad_{g_0}\ln \rho ) ,  \qquad \zeta = \frac{\si^2}{\rho^2}\zeta_0, \qquad \la = \si \la_0 \,.
$$
\end{lemma}

\begin{proof}  The expression for $\mu$ follows by taking the trace in Corollary \ref{cor:B-change}. The Lie bracket is defined independently of the metric and the change in $\zeta$ follows. 
The expression for $\la$ follows since the new horizontal basis is a multiple of $\si$ times the old.
\end{proof}

\begin{lemma} \label{lem:grad}  For a smooth function $f$, 
$$
\grad_gf  = \si^2 \grad_{g_0}f + (\rho^2 - \si^2)\Vv \grad_{g_0}f
$$
\end{lemma}

\begin{proof}
$$
\grad_gf = e_a(f)e_a = \si^2 e_i^0(f)e_i^0 + \rho^2e_r^0(f)e_r^0 = \si^2 \grad_{g_0}f + (\rho^2-\si^2)\Vv \grad_{g_0}f
$$
\end{proof}

Recall that the basis $\{ e_a^0\}$ is chosen such that at the point $x_0$, we have $\Vv \na^0_{e_r^0}e_s^0 = 0, \ \forall r,s = 3,4$. 
\begin{lemma} \label{lem:na-tr} At the point $x_0$, 
$$
\na_{e_a}e_a =  \si^2(2\mu_0 + \Hh \grad{g_0} \ln (\si \la_0\rho^2) ) + \rho^2 \Vv \grad_{g_0} \ln (\rho\si^2 \la_0^2) 
$$
\end{lemma}
\begin{proof} From Corollary \ref{cor:one}(iv),  
$$
\na_{e_a}e_a  = \grad \ln \la + \Vv \grad \ln \la + 2 \mu + \om_{34}(e_r)Je_r\,.
$$
From Lemma \ref{lem:cov-change}, $\om_{34}(e_r)Je_r = \Vv \grad \ln \rho$.  The formula now follows from Lemmas \ref{lem:bic-mu} and \ref{lem:grad}.  
\end{proof}

Define the vertical Laplacian at a point $x$ with respect to the metric $g$ of a smooth function $f$ by $\Delta^{V}_gf = \Delta^{F}_g(f\vert_F) =  e_r(e_r(f)) - \dd f (\Vv \na_{e_r}e_r )$, where $F = \phi^{-1}\phi (x)$ is the fibre passing through $x$.  Similarly, we have the vertical Laplacian with respect to $g_0$. Note that at the point $x_0$, we have $\Delta^{V}_{g_0}f = e_r^0(e_r^0(f))$.  
\begin{lemma} \label{lem:Dan} {\cite{Da-1, Da-2}} 
\begin{eqnarray*}
\Delta_g f & = & \si^2\Delta_{g_0} f  + (\rho^2 - \si^2)\{ \Delta_{g_0}^{V}f -  2\dd f(\Vv \grad_{g_0} \ln \la_0 )\} \\
 & & \qquad - 2\si^2 \dd f(\Hh \grad_{g_0} \ln \rho )- 2\rho^2 \dd f (\Vv \grad_{g_0} \ln \si )
\end{eqnarray*}
\end{lemma}
\begin{remark}  Note that if $\si \equiv \rho$, so that the transformation is conformal, we obtain the well-known formula for the transformation of the Laplacian:
$$
\Delta_g f = \si^2\Delta_{g_0} f  - (m-2)\si^2 \dd f( \grad_{g_0} \ln \si )
$$
(with dimension $m = 4$).  
\end{remark}
\begin{proof}
From Lemma \ref{lem:na-tr}, 
\begin{eqnarray*}
\Delta_g f  & = &  e_a(e_a(f)) - \dd f(\na_{e_a}e_a) =  e_i(e_i(f)) + e_r(e_r(f)) \\
& &  - \dd f\left(2\si^2\mu_0 + \si^2\Hh \grad{g_0} \ln \si \la_0\rho^2  + \rho^2 \Vv \grad_{g_0} \ln \rho\si^2 \la_0^2 \right) \\
 & = & \si^2e_i^0(e_i^0(f)) + \si^2e_i^0(\ln \si )e_i^0(f) + \rho^2e_r^0(e_r^0(f)) + \rho^2e_r^0(\ln \rho ) e_r^0(f) \\
  & & \quad  - \dd f\left(2\si^2\mu_0 + \si^2\Hh \grad{g_0} \ln (\si \la_0\rho^2)  + \rho^2 \Vv \grad_{g_0} \ln \rho\si^2\la_0^2 \right) \\
   & = & \si^2\Delta_{g_0}f + (\rho^2-\si^2)e_r^0(e_r^0(f)) + \si^2\dd f(\na^0_{e_a^0}e_a^0) + \si^2 \dd f(\Hh \grad_{g_0}\ln \si ) + \rho^2 \dd f(\Vv \grad_{g_0}\ln\rho ) \\
    & & \quad - 2\si^2\dd f(\mu_0) - \si^2\dd f(\Hh \grad{g_0} \ln (\si \la_0\rho^2))  - \rho^2 \dd f(\Vv \grad_{g_0} \ln \rho\si^2 \la_0^2) 
\end{eqnarray*}
But from Corollary \ref{cor:one}(iv), $\na^0_{e_a^0}e_a^0 = \grad_{g_0}\ln\la_0 + \Vv \grad_{g_0}\ln \la_0 + 2\mu_0$, so that
\begin{eqnarray*}
\Delta_g f  & = & \si^2\Delta_{g_0}f + (\rho^2-\si^2)e_r^0(e_r^0(f)) + \si^2\dd f(\grad_{g_0}\ln\la_0) +  \si^2\dd f(\Vv \grad_{g_0}\ln\la_0) \\
& & \quad + \si^2 \dd f(\Hh \grad_{g_0}\ln \si ) 
 + \rho^2 \dd f(\Vv \grad_{g_0}\ln\rho ) 
 - \si^2\dd f(\Hh \grad{g_0} \ln (\si \la_0\rho^2)) \\
 & &  \qquad  - \rho^2 \dd f(\Vv \grad_{g_0} \ln \rho\si^2 \la_0^2) \\
    & = &  \si^2\Delta_{g_0}f + (\rho^2-\si^2)\{\Delta_{g_0}^{\Vv}f -2\dd f(\Vv\grad_{g_0}\ln\la_0)\} \\
    & & \qquad 
 - 2\si^2\dd f(\Hh \grad{g_0} \ln \rho )  - 2\rho^2 \dd f(\Vv \grad_{g_0} \ln \si ) 
\end{eqnarray*}
\end{proof}
\begin{corollary} \label{cor:Dan}
\begin{eqnarray*}
\Delta_g \ln \la & = & \si^2\Delta_{g_0} \ln (\si \la_0) + (\rho^2 - \si^2)\{ \Delta_{g_0}^{V}(\ln (\si \la_0)) - 2 \dd \ln (\si \la_0)(\Vv \grad_{g_0} \ln \la_0 )\} \\
 & & \qquad - 2\si^2 \dd \ln (\si \la_0)(\Hh \grad_{g_0} \ln \rho )- 2\rho^2 \dd \ln (\si \la_0) (\Vv \grad_{g_0} \ln \si )
 \end{eqnarray*}
\end{corollary}

\subsection{The second fundamental forms and their divergences}

\n 
The vertical components of the Ricci tensor contain the term $\dv B_1$ acting on vertical vectors.   

\begin{lemma} \label{lem:B1}
$$
B_1(e_i,e_r,e_s) = B_1^0(e_i,e_r^0,e_s^0) + \delta_{rs}e_i(\ln \rho )
$$
\end{lemma}
\begin{proof}  This follows from Corollary \ref{cor:B-change}:
\begin{eqnarray*}
B_1(e_i, e_r, e_s) & = & g(e_i, B_{e_r}e_s) = \frac{1}{\si^2}g_0(e_i, B_{e_r}e_s) \\
 & = & g_0(e_i, B^0_{e_r^0}e_s^0 + g_0(e_r^0, e_s^0) \Hh \grad_{g_0}\ln \rho  ) \\
 & = & B_1^0(e_i,e_r^0, e_s^0) + \delta_{rs} e_i(\ln \rho )
\end{eqnarray*}
\end{proof}

\begin{lemma} \label{lem:divB1}
\begin{eqnarray*}
(\dv B_1)(e_r,e_s) & =  &  \si^2 (\dv_0 B_1^0)(e_r^0,e_s^0) - \si^2B_1^0( \Hh \grad{g_0} \ln \rho^2,e_r^0,e_s^0) \\
& &  + \delta_{rs} \si^2\{ 
 \Tr^{H}_{g_0} \na^0 \dd \ln \rho  + 2\dd \ln \rho (\Vv \grad_{g_0}\ln \la_0)  -\dd\ln \rho (2\mu_0 + \Hh \grad{g_0} \ln \rho^2) \}\,.
\end{eqnarray*}
\end{lemma}
\n (Note that $\Tr^{H}_{g_0}\na \dd \ln \rho$ can be written in terms of the Laplacian and the vertical Laplacian).  

\begin{proof}  Applying Lemma \ref{lem:B1} and Lemma \ref{lem:na-tr},
$$
\begin{array}{l}
(\dv B_1)(e_r,e_s)  =   (\na_{e_a}B_1)(e_a,e_r,e_s) \\
  =  e_i(B_1(e_i,e_r,e_s)) - B_1(\na_{e_a}e_a,e_r,e_s) - B_1(e_i, \na_{e_i}e_r,e_s) - B_1(e_i,e_r,\na_{e_i}e_s) \\
  =   \si e_i^0(\si B_1^0(e_i^0,e_r^0,e_s^0) + \si \delta_{rs} e_i^0(\ln \rho )) - \si^2B_1(2\mu_0+\Hh \grad_{g_0}\ln (\si \la_0\rho^2), e_r, e_s) \\ 
\qquad \qquad - B_1(e_i, \na_{e_i}e_r, e_s) - B_1(e_i, e_r, \na_{e_i}e_s) \\
  =  \si^2e_i^0(B_1^0(e_i^0, e_r^0, e_s^0)) + \si^2 e_i^0(\ln \si ) B_1^0(e_i^0, e_r^0, e_s^0) + \delta_{rs}\si^2e_i^0(e_i^0(\ln \rho ))\\
\qquad \qquad  + \delta_{rs}\si^2e_i^0(\ln \si )e_i^0(\ln \rho ) - \si^2B_1(2\mu_0, \Hh \grad_{g_0}\ln (\si \la_0\rho^2), e_r, e_s) \\
\qquad \qquad  - B_1(e_i, g_0(\na^0_{e_i}e_r^0, e_t^0)e_t, e_s) - B_1(e_i, e_r, g_0(\na^0_{e_i}e_s^0, e_t^0) e_t) \\
  =  \si^2\big\{ \dv_0B^0_1(e_r^0, e_s^0) + B^0_1(\na^0_{e_a^0}e_a^0, e_r^0, e_s^0) + B_1^0(e_i^0, \na^0_{e_i^0}e_r^0, e_s^0) + B_1^0(e_i^0, e_r^0, \na^0_{e_i^0}e_s^0) \\
\qquad \qquad + B_1^0(\Hh \grad_{g_0}\ln \si , e_r^0, e_s^0) + \delta_{rs} e_i^0(e_i^0(\ln \rho )) + \delta_{rs}e_i^0(\ln \si )e_i^0(\ln \rho ) \\
\qquad \qquad - B^0_1(2\mu_0 + \Hh \grad_{g_0}\ln (\si \la_0\rho^2), e_r^0, e_s^0) - \delta_{rs}(2\mu_0+\Hh \grad_{g_0}\ln (\si \la_0\rho^2))(\ln \rho ) \\
\qquad \qquad - B^0_1(e_i^0, \na^0_{e_i^0}e_r^0, e_s^0) - B^0_1(e_i^0, e_r^0, \na^0_{e_i^0}e_s^0) - \big(g_0(\na^0_{e_i^0}e_s^0, e_r^0)+ g_0(\na^0_{e_i^0}e_r^0, e_s^0)\big) e_i(\ln \rho ) \big\}\,. 
 \end{array}
 $$
After simplifying and noting that from Corollary \ref{cor:one}(iii)
$$
\Tr_{g_0}^{H}\na^0\dd \ln \rho = - \dd \ln \rho (\na^0_{e_i^0}e_i^0) + e_i^0(e_i^0(\ln \rho )) = - \dd \ln \rho \big( \grad_{g_0}\ln \la_0 + \Vv \grad_{g_0}\ln \la_0\big) + e_i^0(e_i^0(\ln \rho ))\,,
$$
the formula follows. 
\end{proof}

Let us now deal with $\dv B_2$.  As for Lemma \ref{lem:B1}, we have
\begin{lemma} \label{lem:B2}
$$
B_2(e_r,e_s,e_i) = B_2^0(e_r^0,e_s^0, e_i) + \delta_{rs}e_i(\ln \rho )
$$
\end{lemma}

\begin{lemma} \label{lem:divB2}
\begin{eqnarray*}
(\dv B_2)(e_s,e_j)  =  \dv_{g_0}B^0_2(e_s, e_j) + \na^0\dd \ln \rho (e_s, e_j)   - B^0_2(\Vv \grad_{g_0}\ln (\rho^3\si), e_s, e_j) \\
- e_s(\ln (\si\la_0{}^2))e_j(\ln \rho )  +   2e_s(\ln \rho )\mu_0^{\flat}(e_j) + \left(\frac{\si^2}{\rho^2}-1\right) B^0_2(\zeta_0^{\sharp}, e_s, Je_j) + \left(\frac{\si^2}{\rho^2}- 1 \right) \zeta_0(e_s)Je_j(\ln \rho ) 
\end{eqnarray*}
\end{lemma}
\begin{proof} 
$$  
\begin{array}{l}
(\dv B_2)(e_s,e_j)  =   (\na_{e_a}B_2)(e_a,e_s,e_j) \\
  =   e_r(B_2(e_r, e_s, e_j)) - B_2(\Vv \na_{e_a}e_a, e_s, e_j) - B_2(e_r, \Vv \na_{e_r}e_s, e_j) - B_2(e_r, e_s, \Hh \na_{e_r}e_j)
\end{array}
$$
From Lemma \ref{lem:na-tr}, $\Vv \na_{e_a}e_a = \rho^2 \Vv \grad_{g_0} \ln \rho\si^2 \la_0^2$. From Lemma \ref{lem:cov-change}(iv)
$$
\Vv \na_{e_r}e_s = g(\na_{e_r}e_s, e_t)e_t = (e_t(\ln \rho ) \delta_{rs} - e_s(\ln \rho )\delta_{rt})e_t = \delta_{rs}\Vv \grad \ln \rho - e_s(\ln \rho ) e_r,
$$
and from Lemma \ref{lem:one}(iii), $\Hh \na_{e_r}e_j = - \zeta (e_r)Je_j$.  
Thus, from Lemma \ref{lem:B2},
$$
\begin{array}{l} 
(\dv B_2)(e_s,e_j)  =    e_r(B_2(e_r, e_s, e_j)) -  B_2(\rho^2 \Vv \grad_{g_0} \ln \rho\si^2 \la_0{}^2, e_s, e_j)  \\
 \qquad \qquad - B_2(e_r, \delta_{rs}\Vv \grad \ln \rho - e_s(\ln \rho ) e_r, e_j) +\zeta (e_r) B_2(e_r, e_s, Je_j) \\
 =  e_r(B_2(e_r, e_s, e_j)) -  B_2(\rho^2 \Vv \grad_{g_0} \ln \rho^2\si^2 \la_0{}^2, e_s, e_j)   + 2\mu^{\flat}(e_j)e_s(\ln \rho )  +\zeta (e_r) B_2(e_r, e_s, Je_j)   \\
 =   \rho e_r^0(B^0_2(e_r^0, e_s^0, e_j)) + e_s(e_j(\ln \rho )) - \rho B^0_2(\Vv \grad_{g_0}\ln (\rho^2\si^2\la_0{}^2), e_s^0, e_j)  - e_s(\ln (\rho^2\si^2\la_0{}^2))e_j(\ln \rho ) \\
\qquad \qquad   + 2e_s(\ln \rho )(\mu_0^{\flat}(e_j) + e_j(\ln \rho )) + \frac{\si^2}{\rho^2}\zeta_0(e_r)(B^0_2(e_r^0, e_s^0, Jej) + \delta_{rs}Je_j(\ln \rho )) \\
=   \rho e_r^0(B^0_2(e_r^0, e_s^0, e_j)) + e_s(e_j(\ln \rho )) - \rho B^0_2(\Vv \grad_{g_0}\ln (\rho^2\si^2\la_0{}^2), e_s^0, e_j) \\
\qquad \qquad  - e_s(\ln (\si^2\la_0{}^2))e_j(\ln \rho ) + 2e_s(\ln \rho )\mu_0^{\flat}(e_j) + \frac{\si^2}{\rho^2}B^0_2(\zeta_0^{\sharp}, e_s, Jej) + \frac{\si^2}{\rho^2}\zeta_0(e_s)Je_j(\ln \rho )) \\
  =   \dv_{g_0}B^0_2(e_s, e_j) + B_2^0(\Vv \na^0_{e_a^0}e_a^0, e_s, e_j) + B_2^0(e_r^0, e_s, \Hh \na^0_{e_r^0}e_j) \\
 \qquad \qquad  + e_s(e_j(\ln \rho )) - \rho B^0_2(\Vv \grad_{g_0}\ln (\rho^2\si^2\la_0{}^2), e_s^0, e_j)  - e_s(\ln (\si^2\la_0{}^2))e_j(\ln \rho ) \\
 \qquad \qquad  + 2e_s(\ln \rho )\mu_0^{\flat}(e_j) + \frac{\si^2}{\rho^2}B^0_2(\zeta_0^{\sharp}, e_s, Jej) + \frac{\si^2}{\rho^2}\zeta_0(e_s)Je_j(\ln \rho )) \\
   =   \dv_{g_0}B^0_2(e_s, e_j) + B_2^0(\Vv \na^0_{e_a^0}e_a^0, e_s, e_j) + B_2^0(e_r^0, e_s, - \zeta_0(e_r^0)Je_j + e_r^0(\ln \si )e_j) \\
 \qquad  \qquad + e_s(e_j(\ln \rho )) - B^0_2(\Vv \grad_{g_0}\ln (\rho^2\si^2\la_0{}^2), e_s, e_j)  - e_s(\ln (\si^2\la_0{}^2))e_j(\ln \rho ) \\
\qquad \qquad  + 2e_s(\ln \rho )\mu_0^{\flat}(e_j) + \frac{\si^2}{\rho^2}B^0_2(\zeta_0^{\sharp}, e_s, Jej) + \frac{\si^2}{\rho^2}\zeta_0(e_s)Je_j(\ln \rho )) \\
  =   \dv_{g_0}B^0_2(e_s, e_j) + 2B_2^0(\Vv\grad_{g_0}\ln \la_0, e_s, e_j) - B_2^0(\zeta_0^{\sharp}, e_s, Je_j)  + B_2^0(\Vv \grad_{g_0}\ln \si , e_s, e_j) \\
 \qquad  \qquad + e_s(e_j(\ln \rho )) - B^0_2(\Vv \grad_{g_0}\ln (\rho^2\si^2\la_0{}^2), e_s, e_j)  - e_s(\ln (\si^2\la_0{}^2))e_j(\ln \rho ) \\
\qquad \qquad + 2e_s(\ln \rho )\mu_0^{\flat}(e_j) + \frac{\si^2}{\rho^2}B^0_2(\zeta_0^{\sharp}, e_s, Jej) + \frac{\si^2}{\rho^2}\zeta_0(e_s)Je_j(\ln \rho )) \\
  =   \dv_{g_0}B^0_2(e_s, e_j)   + e_s(e_j(\ln \rho )) - B^0_2(\Vv \grad_{g_0}\ln (\rho^2\si), e_s, e_j)  - e_s(\ln (\si^2\la_0{}^2))e_j(\ln \rho ) \\
\qquad \qquad + 2e_s(\ln \rho )\mu_0^{\flat}(e_j) + \left(\frac{\si^2}{\rho^2}-1\right) B^0_2(\zeta_0^{\sharp}, e_s, Jej) + \frac{\si^2}{\rho^2}\zeta_0(e_s)Je_j(\ln \rho ) 
\end{array}
$$
However
\begin{eqnarray*}
\na^0\dd \ln \rho (e_s, e_j) & = &  - (\na^0_{e_s}e_j)(\ln \rho ) + e_s(e_j(\ln \rho )) \\
 & = & - (\Hh \na^0_{e_s}e_j)(\ln \rho ) - (\Vv \na^0_{e_s}e_j)  (\ln \rho ) + e_s(e_j(\ln \rho ) \\
 & = & \zeta_0(e_s)Je_j(\ln \rho ) - e_s(\ln \si ) e_j(\ln \rho ) + B^{0*}_{e_s}e_j(\ln \rho ) + e_s(e_j(\ln \rho )) 
\end{eqnarray*}
Finally,
$$
B^{0*}_{e_s}e_j(\ln \rho ) = g_0(\Vv \grad \ln \rho ,B^{0*}_{e_s}e_j) =  g_0(B^0_{e_s}\Vv \grad \ln \rho , e_j) = B_2^0(\Vv \grad \ln \rho , e_s, e_j)
$$
and the expression follows. 
\end{proof}

\begin{lemma} \label{lem:conf-C} Under biconformal deformation, the quantities $C$ and $C^*$ change according to
\begin{eqnarray*}
C & = & \frac{\si^2}{\rho^2}\left\{ C_0 + \dd \ln \rho (B^0_{\star}\star) +  ||\Hh \grad_{g_0}\ln \rho ||^2_{g_0} g_0^{\Vv} \right\}  \\
C^* & = & C_0^* + 4\dd \ln \rho \odot \mu_0^{\flat} + 2(\dd \ln \rho\circ \Hh )^2
\end{eqnarray*}  
\end{lemma}
\begin{proof} From Corollary \eqref{cor:B-change},
\begin{eqnarray*}
C(e_r, e_s) & = & g(B_{e_t}e_r, B_{e_t}e_s) = \frac{1}{\si^2} g_0(B_{e_t}e_r, B_{e_t}e_s) \\
 & = & \frac{1}{\si^2} g_0( \si^2(B^0_{e_t^0}e_r^0 + \delta_{rt} \Hh \grad_{g_0} \ln \rho ) , \si^2(B^0_{e_t^0}e_s^0 + \delta_{st} \Hh \grad_{g_0} \ln \rho ) ) \\
 & = & \frac{\si^2}{\rho^2}C_0(e_r, e_s) + \frac{2\si^2}{\rho^2} \dd \ln \rho (B^0_{e_r}e_s) + \frac{\si^2}{\rho^2} ||\Hh \grad_{g_0}\ln \rho ||^2_{g_0} g_0(e_r,e_s)\,.
\end{eqnarray*}
Whereas
\begin{eqnarray*}
C^*(e_i, e_j) & = & g(B^*_{e_r}e_i, B^*_{e_r}e_j) = g(e_s, B^*_{e_r}e_i)g(e_s, B^*_{e_r}e_j) = g(B_{e_r}e_s, e_i)g(B_{e_r}e_s,e_j) \\
& = & \frac{1}{\si^4} g_0(\si^2(B^0_{e_r^0}e_s^0 + \delta_{rs}\Hh \grad_{g_0} \ln \rho ), e_i)g_0(\si^2(B^0_{e_r^0}e_s^0 + \delta_{rs}\Hh \grad_{g_0} \ln \rho ), e_j) \\
& = & C_0^*(e_i, e_j) + 2\dd \ln \rho (e_i)\mu_0^{\flat}(e_j) +2\dd \ln \rho (e_j)\mu_0^{\flat}(e_i) + 2 \dd \ln \rho (e_i)\dd \ln \rho (e_j)\,.
\end{eqnarray*}
\end{proof}

\begin{remark} \label{rmk:conformal-ricci} 
When $\si = \rho$, the deformation is conformal and there is a well-known formula for the change in Ricci \cite{He}:
\begin{eqnarray*}
\Ric (e_a, e_b) & = & \Ric^0 (e_a, e_b)+ 2[\na^0 \dd \ln \si (e_a, e_b)+ e_a(\ln \si)e_b(\ln \si )] \\ & & \qquad + (\Delta_{g_0} \ln \si - 2 || \grad_{g_0} \ln \si ||^2)g_0(e_a, e_b)\,.
\end{eqnarray*} 
\end{remark}

\section{Orthogonal projection from $\RR^4$ to $\RR^2$} \label{sec:can-proj}
\n Let $\phi : \RR^4 \ra \RR^2$ be the canonical projection $\phi (x^1, x^2, x^3, x^4 ) = (x^1, x^2)$.  Then $\la_0 \equiv 1, \mu_0 \equiv 0, B^0 = B_1^0=B_2^0 \equiv 0, \zeta^0 \equiv 0$.  We take the standard basis: $e_a^0 = \pa / \pa x^a$.  

From Lemma \ref{lem:ricci-hor}, 
\begin{eqnarray*}
\Ric\vert_{H\times H} & =  & \left\{ \la^2K^N + \Delta \ln \la + 2\dd \ln \la (\mu ) - 2 ||\zeta ||^2 \right\} g^{\Hh} 
-C^* + \Ll_{\mu}g \vert_{H \times H} \\
 & = & \left\{ \Delta \ln \la + 2\dd \ln \la (\mu ) \right\} g^{\Hh} 
- C^* + \Ll_{\mu}g \vert_{H \times H}\,,
\end{eqnarray*}
where $\la = \si$ and $\mu = \si^2 \Hh \grad_{g_0}\ln \rho$.  

From Corollary \ref{cor:Dan}, 
$$
\Delta_g \ln \la  =   \si^2\Delta_{g_0} \ln \si  + (\rho^2 - \si^2)\Delta_{g_0}^{\Vv}(\ln \si )
- 2\si^2 \dd \ln \si (\Hh \grad_{g_0} \ln \rho )- 2\rho^2 \dd \ln \si  (\Vv \grad_{g_0} \ln \si ) \\
$$
and $\dd \ln \la (\mu )  = \si^2\dd\ln \si (\Hh \grad_{g_0}\ln \rho )$, so that 
$$
\Delta_g \ln \la  + 2 \dd \ln \la (\mu ) = \si^2\left( \frac{\pa^2\ln \si}{\pa x_1{}^2} +  \frac{\pa^2\ln \si}{\pa x_2{}^2}\right) + \rho^2 \left( \frac{\pa^2\ln \si}{\pa x_3{}^2} +  \frac{\pa^2\ln \si}{\pa x_4{}^2}\right) - 2\rho^2 \left( \left( \frac{\pa \ln \si}{\pa x_3}\right)^2 + \left( \frac{\pa \ln \si}{\pa x_4}\right)^2\right)\,. 
$$
From Lemma \ref{lem:conf-C},  
$$
C^*(e_i, e_j) = 2 \si^2 e_i^0(\ln \rho ) e_j^0(\ln \rho ) = 2\si^2 \frac{\pa \ln \rho}{\pa x^i}\frac{\pa \ln \rho}{\pa x^j}\,.
$$
Next, from Lemma \ref{lem:cov-change}(v),
\begin{eqnarray*}
\Ll_{\mu}g(e_i, e_j) & = & g(\na_{e_i}\mu, e_j) + g(e_i, \na_{e_j}\mu ) \\
& = & e_i(g(\mu , e_j)) + e_j(g(\mu , e_i)) - g(\mu , \na_{e_i}e_j+ \na_{e_j}e_i) \\
 & = & e_i(e_j(\ln \rho )) + e_j(e_i(\ln \rho ))  -   e_k(\ln \rho )g(e_k, \na_{e_i}e_j+\na_{e_j}e_i) \\
 & = & e_i(e_j(\ln \rho )) + e_j(e_i(\ln \rho ))  + e_i(\ln \rho )e_j(\ln \si ) + e_i(\ln \si )e_j(\ln \rho ) - 2 \delta_{ij} \Hh \grad_g\ln \rho \\
 & = & 2 \si^2 \bigg\{ \frac{\pa^2\ln \rho}{\pa x^i\pa x^j} + \frac{\pa \ln \si}{\pa x^i}\frac{\pa \ln \rho}{\pa x^j} + \frac{\pa \ln \rho}{\pa x^i}\frac{\pa \ln \si}{\pa x^j} - \delta_{ij} \left( \frac{\pa \ln \si}{\pa x_1}\frac{\pa \ln \rho}{\pa x_1}+ \frac{\pa \ln \si}{\pa x_2}\frac{\pa \ln \rho}{\pa x_2}\right) \bigg\}\,.
 \end{eqnarray*}
 
 Collecting terms, we obtain
$$
\begin{array}{lll} 
 \Ric (e_1, e_1) & = &  \si^2\bigg\{ \left( \frac{\pa^2\ln \si}{\pa x_1{}^2} +  \frac{\pa^2\ln \si}{\pa x_2{}^2}\right) + \frac{\rho^2}{\si^2} \left( \frac{\pa^2\ln \si}{\pa x_3{}^2} +  \frac{\pa^2\ln \si}{\pa x_4{}^2}\right) - 2\frac{\rho^2}{\si^2}  \left( \left( \frac{\pa \ln \si}{\pa x_3}\right)^2 + \left( \frac{\pa \ln \si}{\pa x_4}\right)^2\right) \\
 & & \quad - 2\left( \frac{\pa \ln \rho}{\pa x_1}\right)^2 + 2\frac{\pa^2\ln \rho}{\pa x_1^2} + 2\frac{\pa \ln \si}{\pa x_1}\frac{\pa \ln \rho}{\pa x_1} - 2\frac{\pa \ln \si}{\pa x_2}\frac{\pa \ln \rho}{\pa x_2}\bigg\} \\
 \Ric (e_2, e_2) & = &  \si^2\bigg\{ \left( \frac{\pa^2\ln \si}{\pa x_1{}^2} +  \frac{\pa^2\ln \si}{\pa x_2{}^2}\right) + \frac{\rho^2}{\si^2} \left( \frac{\pa^2\ln \si}{\pa x_3{}^2} +  \frac{\pa^2\ln \si}{\pa x_4{}^2}\right) - 2\frac{\rho^2}{\si^2}  \left( \left( \frac{\pa \ln \si}{\pa x_3}\right)^2 + \left( \frac{\pa \ln \si}{\pa x_4}\right)^2\right) \\
 & & \quad - 2\left( \frac{\pa \ln \rho}{\pa x_2}\right)^2 + 2\frac{\pa^2\ln \rho}{\pa x_2^2} + 2\frac{\pa \ln \si}{\pa x_2}\frac{\pa \ln \rho}{\pa x_2} - 2\frac{\pa \ln \si}{\pa x_1}\frac{\pa \ln \rho}{\pa x_1}\bigg\} \\
 \Ric (e_1, e_2) & = &  2\si^2\bigg\{ \frac{\pa^2\ln \rho}{\pa x_1\pa x_2} - \frac{\pa \ln \rho}{\pa x_1} \frac{\pa \ln \rho}{\pa x_2} + \frac{\pa \ln \si}{\pa x_1}\frac{\pa \ln \rho}{\pa x_2} + \frac{\pa \ln \rho}{\pa x_1}\frac{\pa \ln \si}{\pa x_2} \bigg\}\,.
  \end{array}   
$$

 From Lemma \ref{lem:ric-mixed}, the mixed Ricci tensor acting on $(e_j, e_s)$ is given by 
\begin{eqnarray*}
\Ric (e_j,e_s) & = & \na \dd \ln \la (e_j,e_s) - (\dd \ln \la )^2(e_j,e_s) - 2 (\dd \ln \la \odot \zeta)(Je_j, e_s) - (\na_{Je_j}\zeta )(e_s) - 2 \zeta (\na_{e_s} Je_j)  \\
 & & \qquad - \dv B_2(e_s,e_j) - 2\dd\ln \la (B^*_{e_s}e_j)  + 2 (\na_{e_s}\mu^{\flat})(e_j) \\
 & =  & \na \dd \ln \si (e_j,e_s) - (\dd \ln \si )^2(e_j,e_s)  - \dv B_2(e_s,e_j) - 2\dd\ln \la (B^*_{e_s}e_j) + 2 (\na_{e_s}\mu^{\flat})(e_j) 
\end{eqnarray*}
From Corollary \ref{cor:conn}, 
$$
\na \dd \ln \si (e_s, e_j) = e_s(e_j(\ln \si )) - \dd \ln \si (\na_{e_s}e_j) = e_s(e_j(\ln \si )) + e_j(\ln \rho ) e_s(\ln \si )
$$
Since the fibres before deformation are totally geodesic, $B^0 \equiv 0$, so from Lemma \ref{lem:divB2}, 
\begin{eqnarray*}
(\dv B_2)(e_s, e_j) & = &  \na^0 \dd\ln \rho (e_s, e_j) - e_s(\ln \si )e_j(\ln \rho ) = e_s(e_j(\ln \rho )) -\dd \ln \rho (\na^0_{e_s}e_j)- e_s(\ln \si )e_j(\ln \rho ) \\
 & = & e_s(e_j(\ln \rho )) -\dd \ln \rho \big(\si \rho \na^0_{e_s^0}e_j^0 + \si \rho e_s^0(\ln \si )e_j^0\big)- e_s(\ln \si )e_j(\ln \rho ) \\
 &= & e_s(e_j(\ln \rho )) - 2e_s(\ln \si ) e_j(\ln \rho )
\end{eqnarray*}
From Corollary \eqref{cor:conn}, 
$$
\dd \ln \la (B^*_{e_s}e_j) = - \dd \ln \si (\Vv \na_{e_s}e_j) = - \dd \ln \si (e_r)g(e_r, \na_{e_s}e_j) = e_s(\ln \si )e_j(\ln \rho )\,.
$$
Finally, $\mu = \si^2 \Hh \grad_{g_0}\ln \rho = e_i(\ln \rho ) e_i$, so that from Corollary \ref{cor:conn},   
\begin{eqnarray*}
(\na_{e_s}\mu^{\flat})(e_j) & = & e_s(g(\mu , e_j)) - g(\mu , \na_{e_s}e_j) \\
 & = & e_s(e_i(\ln \rho )\delta_{ij}) - e_i(\ln \rho ) g(e_i, \na_{e_s}e_j) = e_s(e_j(\ln \rho ))
\end{eqnarray*}
We conclude that
$$
\Ric (e_j, e_s) = e_s(e_j(\ln \si )) + e_s(e_j(\ln \rho ))+ e_j(\ln \rho ) e_s(\ln \si ) - e_j(\ln \si) e_s(\ln \si ),  
$$
explicitly
$$
\Ric (e_j, e_s)  =  \si \rho \bigg\{ \frac{\pa^2 \ln (\si \rho )}{\pa x^j\pa x^s} + 2 \frac{\pa \ln \si}{\pa x^s}\frac{\pa \ln \rho}{\pa x^j}\bigg\} \,.
$$

 The vertical components of the Ricci tensor are given by
 \begin{eqnarray*}
\Ric\vert_{V\times V}  & = & K^{V}g^{V} + 2 \na \dd \ln\la\vert_{V\times V} + 2\dd \ln \la (B_{\star}\star ) - 2 (\dd \ln \la )^2\vert_{V\times V} + 2 \zeta^2 +\dv B_1\vert_{V\times V}   \\
& = &K^{V}g^{V} + 2 \na \dd \ln\la\vert_{V\times V} + 2\dd \ln \la   (B_{\star}\star ) - 2 (\dd \ln \la )^2\vert_{V\times V}  +\dv B_1\vert_{V\times V}  
\end{eqnarray*} 
After biconformal deformation, the sectional curvature of the fibres is given by
$$
K^{V} =  \rho^2 \Delta_{g_0}^{V} \ln \rho 
$$
For the second fundamental form:
$$
\na \dd \ln \la (e_r, e_s) = e_r(e_s(\ln \la )) - \dd \ln \la (\na_{e_r}e_s)
$$
From Corollary \ref{cor:conn},
$$
\na_{e_r}e_s = \delta_{rs}\{ \si^2\Hh\grad_{g_0}\ln \rho + \rho^2\Vv \grad_{g_0}\ln \rho\} - e_s(\ln \rho )e_r
$$
and
\begin{eqnarray*}
\na \dd \ln \la (e_r, e_s) & = & e_r(e_s(\ln \la )) - \dd \ln \la (\na_{e_r}e_s)\\
& = & \rho^2 e_r^0(e_s^0(\ln \si  )) + \rho^2e_r^0(\ln \rho )e_s^0(\ln \si  )  + \rho^2 e_s^0(\ln \rho ) e_r^0(\ln \si ) \\
 & & \quad - \delta_{rs} \dd \ln \si  (\si^2 \Hh \grad_{g_0}\ln \rho + \rho^2\Vv \grad_{g_0}\ln \rho ) 
\end{eqnarray*}
From Corollary \eqref{cor:B-change},
$$
B_{e_r}e_s = \si^2 \delta_{rs}\Hh \grad_{g_0}\ln \rho\,. 
$$
From Lemma \ref{lem:divB1} we have
$$
(\dv B_1)(e_r,e_s) =    \delta_{rs} \si^2\{ 
 \Tr^{\Hh}_{g_0} \na \dd \ln \rho    -\dd\ln \rho (\Hh \grad_{g_0} \ln \rho^2) \}
 $$
 Thus 
\begin{eqnarray*}
 \Ric (e_r,e_s) & = & \rho^2 \delta_{rs} \Delta_{g_0}^{V} \ln \rho - 2\rho^2 e_r^0(\ln \si ) e_s^0(\ln \si ) \\
  & & + 2\{ \rho^2 e_r^0(e_s^0(\ln \si  )) + \rho^2e_r^0(\ln \rho )e_s^0(\ln \si  )  + \rho^2 e_s^0(\ln \rho ) e_r^0(\ln \si )\} \\
 & &   + \delta_{rs} \{ 
 \si^2\Tr^{\Hh}_{g_0} \na \dd \ln \rho  - 2 \si^2 \dd \ln \rho (\Hh \grad_{g_0} \ln \rho ) - 2 \rho^2\dd \ln \si  ( \Vv \grad_{g_0}\ln \rho )  \} 
 \end{eqnarray*} 
Explicitly,
 $$
 \begin{array}{lll}
 \Ric (e_r,e_s) & = & \rho^2\bigg\{ 2\frac{\pa^2\ln \si}{\pa x^r\pa x^s} +2\frac{\pa \ln \rho}{\pa x^r}\frac{\pa \ln \si}{\pa x^s} + 2\frac{\pa \ln \si}{\pa x^r} \frac{\pa \ln \rho}{\pa x^s} - 2\frac{\pa \ln \si}{\pa x^r}\frac{\pa \ln \si}{\pa x^s} \\
& &  + \delta_{rs} \bigg( \frac{\si^2}{\rho^2}\left( \frac{\pa^2\ln \rho}{\pa x_1{}^2} + \frac{\pa^2\ln \rho}{\pa x_2{}^2} \right)  + \frac{\pa^2\ln \rho}{\pa x_3{}^2} + \frac{\pa^2\ln \rho}{\pa x_4{}^2} - 2\frac{\si^2}{\rho^2}\left( \left(\frac{\pa \ln \rho}{\pa x_1}\right)^2 + \left(\frac{\pa \ln \rho}{\pa x_2}\right)^2 \right)   \\
& & \qquad - 2\left( \frac{\pa \ln \si}{\pa x_3}\frac{\pa \ln \rho}{\pa x_3} + \frac{\pa \ln \si}{\pa x_4}\frac{\pa \ln \rho}{\pa x_4}\right)     \bigg)\bigg\}  
  \end{array} 
 $$
 
 The equations for an Einstein metric: $\Ric = Ag$ for some constant $A$, become the following system of ten equations: 
 \begin{equation} \label{Einstein-proj}
 \begin{array}{llll}
 {\rm (i)} & A & = & \si^2\bigg\{ \frac{\pa^2\ln \si}{\pa x_1{}^2} +  \frac{\pa^2\ln \si}{\pa x_2{}^2} + \frac{\rho^2}{\si^2} \left( \frac{\pa^2\ln \si}{\pa x_3{}^2} +  \frac{\pa^2\ln \si}{\pa x_4{}^2}\right) - 2\frac{\rho^2}{\si^2}  \left( \left( \frac{\pa \ln \si}{\pa x_3}\right)^2 + \left( \frac{\pa \ln \si}{\pa x_4}\right)^2\right) \\
 & & &  \quad + 2\frac{\pa^2\ln \rho}{\pa x_j^2} - 2\left( \frac{\pa \ln \rho}{\pa x_j}\right)^2 + + 2\frac{\pa \ln \si}{\pa x_j}\frac{\pa \ln \rho}{\pa x_j} - 2\frac{\pa \ln \si}{\pa x_{j'}}\frac{\pa \ln \rho}{\pa x_{j'}}\bigg\} \quad (j = 1,2) \\
 {\rm (ii)} & 0 & = &  \frac{\pa^2\ln \rho}{\pa x_1\pa x_2} + \frac{\pa \ln \si}{\pa x_1}\frac{\pa \ln \rho}{\pa x_2} + \frac{\pa \ln \rho}{\pa x_1}\frac{\pa \ln \si}{\pa x_2} - \frac{\pa \ln \rho}{\pa x_1} \frac{\pa \ln \rho}{\pa x_2}  \\
 {\rm (iii)} &  0 & = & \frac{\pa^2 \ln (\si \rho )}{\pa x^j\pa x^s} + 2 \frac{\pa \ln \si}{\pa x^s}\frac{\pa \ln \rho}{\pa x^j} \quad (j = 1,2,\, s = 3,4) \\
 {\rm (iv)} & A & = & \rho^2\bigg\{ 2\frac{\pa^2\ln \si}{\pa x_s{}^2}  - 2\left(\frac{\pa \ln \si}{\pa x_s}\right)^2+ 2\frac{\pa \ln \si}{\pa x_s} \frac{\pa \ln \rho}{\pa x_s} - 2\frac{\pa \ln \si}{\pa x_{s'}}\frac{\pa \ln \rho}{\pa x_{s'}}  \quad (s = 3,4) \\
& & & +  \frac{\si^2}{\rho^2}\left( \frac{\pa^2\ln \rho}{\pa x_1{}^2} + \frac{\pa^2\ln \rho}{\pa x_2{}^2} \right)  + \frac{\pa^2\ln \rho}{\pa x_3{}^2} + \frac{\pa^2\ln \rho}{\pa x_4{}^2} - 2\frac{\si^2}{\rho^2}\left( \left(\frac{\pa \ln \rho}{\pa x_1}\right)^2 + \left(\frac{\pa \ln \rho}{\pa x_2}\right)^2 \right)   \bigg\}  \\
{\rm (v)} & 0 & = & \frac{\pa^2\ln \si}{\pa x_3\pa x_4} +\frac{\pa \ln \rho}{\pa x_3}\frac{\pa \ln \si}{\pa x_4} + \frac{\pa \ln \si}{\pa x_3} \frac{\pa \ln \rho}{\pa x_4} - \frac{\pa \ln \si}{\pa x_3}\frac{\pa \ln \si}{\pa x_4} 
 \end{array}
 \end{equation} 

Note the symmetry between the equations: after the interchange $(j,k,\si , \rho ) \leftrightarrow (r,s,\rho , \si )$, equations (i) and (iv) are interchanged. 

\subsection{Warped product solutions} \label{sec:warped-product} 

\n Let us investigate some special solutions.  If $\si = \si (x_1, x_2)$ and $\rho = \rho (x_3,x_4)$, then the system reduces to 
$$
\frac{\pa^2\ln \si}{\pa x_1{}^2} +  \frac{\pa^2\ln \si}{\pa x_2{}^2} = A/\si^2 \qquad {\rm and} \qquad  \frac{\pa^2\ln \rho}{\pa x_3{}^2} + \frac{\pa^2\ln \rho}{\pa x_4{}^2} = A/\rho^2\,.
$$
Note that $A = \si^2\left(\frac{\pa^2\ln \si}{\pa x_1{}^2} +  \frac{\pa^2\ln \si}{\pa x_2{}^2}\right)$ is the Gaussian curvature of the surface $(\RR^2, (\dd x_1{}^2 + \dd x_2{}^2)/\si^2)$; similarly for the second equation.   
For example, setting 
$$
\si = \frac{1 + x_1{}^2 + x_2{}^2}{2} \qquad {\rm and} \qquad \rho = \frac{1+x_3{}^2 + x_4{}^2}{2}
$$
yields the product of spheres $S^2 \times S^2$ with constant $A = 1$, whereas setting  
$$
\si = \frac{1 - x_1{}^2 - x_2{}^2}{2} \qquad {\rm and} \qquad \rho = \frac{1-x_3{}^2 - x_4{}^2}{2}
$$
yields the product of hyperbolic spaces $H^2 \times H^2$ with constant $A = -1$. 

More generally a warped product of the surfaces $(\RR^2, (\dd x_1{}^2+\dd x_2{}^2)/\si (x_1, x_2)^2)$ and $(\RR^2, (\dd x_3{}^2+\dd x_4{}^2)/\be (x_3, x_4)^2)$ corresponds to $\RR^4$ endowed with a metric of the form: 
\begin{equation} \label{metric-warped}
g = \frac{\dd x_1{}^2 + \dd x_2{}^2}{\si (x_1, x_2)^2} + \frac{\dd x_3{}^2 + \dd x_4{}^2}{\al (x_1, x_2)^2 \be (x_3, x_4)^2}\,.
\end{equation} 
Setting $\si = \si (x_1, x_2)$ and $\rho = \al (x_1, x_2)\be (x_3, x_4)$, the Einstein equations become the system: 
\begin{equation} \label{Einstein-warped}
 \begin{array}{llll}
 {\rm (i)} & A & = & \si^2\bigg\{ \frac{\pa^2\ln \si}{\pa x_1{}^2} +  \frac{\pa^2\ln \si}{\pa x_2{}^2} - 2\left( \frac{\pa \ln \al}{\pa x_1}\right)^2 + 2\frac{\pa^2\ln \al}{\pa x_1^2} + 2\frac{\pa \ln \si}{\pa x_1}\frac{\pa \ln \al}{\pa x_1} - 2\frac{\pa \ln \si}{\pa x_2}\frac{\pa \ln \al}{\pa x_2}\bigg\} \\
 {\rm (ii)} & A & = & \si^2\bigg\{  \frac{\pa^2\ln \si}{\pa x_1{}^2} +  \frac{\pa^2\ln \si}{\pa x_2{}^2} - 2\left( \frac{\pa \ln \al}{\pa x_2}\right)^2 + 2\frac{\pa^2\ln \al}{\pa x_2^2} + 2\frac{\pa \ln \si}{\pa x_2}\frac{\pa \ln \al}{\pa x_2} - 2\frac{\pa \ln \si}{\pa x_1}\frac{\pa \ln \al}{\pa x_1}\bigg\} \\
 {\rm (iii)} & 0 & = &  \frac{\pa^2\ln \al}{\pa x_1\pa x_2}  + \frac{\pa \ln \si}{\pa x_1}\frac{\pa \ln \al}{\pa x_2} + \frac{\pa \ln \al}{\pa x_1}\frac{\pa \ln \si}{\pa x_2} - \frac{\pa \ln \al}{\pa x_1} \frac{\pa \ln \al}{\pa x_2} \\
 {\rm (iv)} & A & = & \si^2\left(  \frac{\pa^2\ln \al}{\pa x_1{}^2} + \frac{\pa^2\ln \al}{\pa x_2{}^2} \right)  + \al^2\be^2\left( \frac{\pa^2\ln \be}{\pa x_3{}^2} + \frac{\pa^2\ln \be}{\pa x_4{}^2} \right) - 2\si^2\left( \left(\frac{\pa \ln \al}{\pa x_1}\right)^2 + \left(\frac{\pa \ln \al}{\pa x_2}\right)^2    \right)  
  \end{array}
 \end{equation} 
 The sum of (i) and (ii) gives the equation:
 $$
 A = \si^2 \left( \Delta_{g_0}\ln\si + \Delta_{g_0}\ln \al - ||\grad_{g_0} \ln \al ||_0^2\right)\,. 
  $$ 
On the other hand the difference gives:
 $$
 0 = \frac{\pa^2\ln \al}{\pa x_1{}^2} - \frac{\pa^2\ln \al}{\pa x_2{}^2} - \left( \frac{\pa \ln \al}{\pa x_1}\right)^2 + \left( \frac{\pa \ln \al}{\pa x_2}\right)^2 + 2 \frac{\pa \ln \si}{\pa x_1}\frac{\pa \ln \al}{\pa x_1} - 2 \frac{\pa \ln \si}{\pa x_2}\frac{\pa \ln \al}{\pa x_2}\,.
 $$  
 Since $\al = \al (x_1, x_2)$ and $\be = \be (x_3, x_4)$ are independent, on dividing equation (iv) by $\al^2$, we deduce that 
\begin{equation}\label{beta} 
 \be^2\left( \frac{\pa^2\ln \be}{\pa x_3{}^2} + \frac{\pa^2\ln \be}{\pa x_4{}^2} \right)  = C
\end{equation}  
 for a constant $C$ and in particular the metric $(\dd x_3{}^2+\dd x_4{}^2)/\be^2$ is necessarily of constant Gaussian curvature $C$, and
 $$
 A - C\al^2 = \si^2 \Delta_{g_0}\ln \al   - 2\si^2||\grad_{g_0} \ln \al ||_0^2\,.
 $$
 
 Set $x_1 = t$ and suppose that $\al = \al (t)$ and $\si = \si (t)$ depend only on $t$.  Then \eqref{Einstein-warped}(iii) is satisfied and $\al$ and $\si$ are determined by the system: 
 \begin{equation} \label{syst-warped}
 \left\{ \begin{array}{lrcl}
{\rm (i)} &   A & = & \si^2 \left( (\ln \si )'' + (\ln \al )'' - (\ln \al )'^2\right)  \\
{\rm (ii)} &   0 & = & (\ln \al )'' - (\ln \al )'^2+2(\ln \si )'(\ln \al )'  \\
{\rm (iii)} &  A-C\al^2 & = &   \si^2 \left( (\ln \al )'' - 2(\ln \al )'^2\right) 
 \end{array} \right.
 \end{equation} 
 From \eqref{syst-warped}(ii), provided $(\ln \al )' \not\equiv 0$,  
 \begin{eqnarray*}
 2(\ln \si )' & = & \frac{- (\ln \al )'' + (\ln \al )'^2}{(\ln \al )'} = (- \ln |(\ln \al )'| + \ln \al )'  \\
  \Longrightarrow  \ 2 \ln \si & = & - \ln |(\ln \al )'| + \ln \al + a \ \Longrightarrow \ \si^2 = B \al^2/\al '\,,
 \end{eqnarray*} 
 for constants $a$ and $B$, with $B$ non-zero. In particular, taking the difference between \eqref{syst-warped}(i) and (iii), we deduce that
 $$
 \wt{C} \al ' = (\ln \si )'' + (\ln \al )'^2\,,
 $$
where $\wt{C}= C/B$.  But from \eqref{syst-warped}(ii), 
 \begin{eqnarray*}
 2 (\ln \si )'' & = & \frac{-(\ln \al )''' +  (\ln \al )'(\ln \al )''}{(\ln \al )'} + \frac{(\ln \al )''^2}{(\ln \al )'^2}  \\
 \Longrightarrow \ 2\wt{C}\al ' & = & \frac{-(\ln \al )''' +  (\ln \al )'(\ln \al )''}{(\ln \al )'} + \frac{(\ln \al )''^2}{(\ln \al )'^2} + 2(\ln \al )'^2 
 \end{eqnarray*} 
 This simplifies to the third order ODE:
 $$
\al\al ''' = 2 \al '\al '' + \frac{\al(\al '')^2}{\al '} - 2\wt{C}\al  \al '^2\,.
 $$
 Note the specific solution $\al (t) = t$ corresponding to hyperbolic space. More generally, 
if we set $\ga (t) = \al '(t),\, \delta (t) = \al '' (t) = \ga '(t)$, then we have the first order system:
\begin{equation} \label{1st-order}
\left( \begin{array}{c} \al \\ \ga \\ \delta \end{array} \right)^{\prime} = \left( \begin{array}{c} \ga \\ \delta \\ \frac{2\ga \delta}{\al} + \frac{\delta^2}{\ga} - 2\wt{C} \ga^2 \end{array} \right)
\end{equation} 
 \emph{Cauchy's existence theorem} (see, for example, \cite{Di} (10.4.5)) yields local solutions: Let $\Ga_0 =  \left( \begin{array}{c} \al_0 \\ \ga_0 \\ \delta_0 \end{array} \right)\in \RR^3$ be a point with $\al_0 > 0$ and $\ga_0\neq 0$ and let $ t_0\in \RR$. Then there is a solution $\Ga (t) = \left( \begin{array}{c} \al (t) \\ \ga (t) \\ \delta (t) \end{array} \right)$ to \eqref{1st-order} defined on an open interval $I \subset \RR$ ($t_0 \in I$), with $\Ga (t_0) = \Ga_0$.  
 
 Given such a solution to \eqref{1st-order} on an open interval $I$ with $\al (t)$ positive and $\al '(t)$ non-zero for all $t \in I$, then defining $\si$ by $\si^2 = B\al^2/\al '$, where $B$ is a non-zero constant of sign consistent with $\si^2>0$ and where we require $C = B\wt{C}$ to be the constant Gaussian curvature of the metric $(\dd x_3{}^2 + \dd x_4{}^2)/\be (x_3, x_4)^2$, equations \eqref{syst-warped} are satisfied and the metric \eqref{metric-warped} is Einstein. The constant $A$ is given by \eqref{syst-warped}(iii): 
 $$
 A = C \al^2 + \frac{B\al^2}{ \al '} \left( \frac{\al ''}{\al} - \frac{3 \al '^2}{\al^2}\right)\,,
 $$
 which one easily checks is an integral of \eqref{1st-order}.

\subsection{Solutions depending on a single parameter} \label{sec:single} 

Replace $x_1$ with the parameter $t$ and suppose that both $\si $ and $\rho$ depend only on $t$.  Then \eqref{Einstein-proj}(iii), (iv) and (vii) are satisfied, while (i) becomes
\begin{equation}\label{proj-1}
A = \si^2 \big\{ (\ln \si )'' + 2 (\ln \si )'(\ln \rho )' - 2 (\ln \rho )'^2 + 2 (\ln \rho )'' \big\} \,;
\end{equation}
(ii)  becomes
\begin{equation} \label{proj-2}
A = \si^2 \big\{ (\ln \si )'' - 2(\ln \si )'(\ln \rho )'\big\}\,; 
\end{equation}
(v) and (vi) become
\begin{equation} \label{proj-3}
A = \si^2\big\{ (\ln \rho )'' - 2 (\ln \rho )'^2\big\}\,. 
\end{equation} 
The first two of these are equivalent to the pair of equations:
\begin{equation} \label{special-0}
\left\{ \begin{array}{crcl}
{\rm (a)} & A & = & \si^2 (\ln \si )'' - 2\si^2 (\ln \si )'(\ln \rho )' \\
{\rm (b)} & 0 & = & - (\ln \rho )'^2 + (\ln \rho )'' + 2(\ln \si )'(\ln \rho )'
\end{array} \right.
\end{equation}
while the third becomes:
\begin{equation} \label{special-1}
A = \si^2 (\ln \rho )'' - 2\si^2 (\ln \rho )'^2 \quad  \stackrel{{\rm (b)}}{\Longrightarrow}  \quad \frac{A}{\si^2} = - (\ln \rho )'^2 -2 (\ln \si )' (\ln \rho )' \quad \stackrel{{\rm (a)}}{\Longrightarrow} - (\ln \rho )'^2 = (\ln \si )'' 
\end{equation}

We can combine the \eqref{special-0}(a) and the first identity of \eqref{special-1} to deduce
\begin{eqnarray*}
(\ln \rho )'' - (\ln \si )'' & = & 2(\ln \rho )'((\ln \rho )' - (\ln \si )') \quad \Rightarrow \quad \left(\ln \left(\frac{\rho}{\si}\right)\right)'' = 2 (\ln \rho )' \left(\ln \left(\frac{\rho}{\si}\right)\right)' \\
 & \Rightarrow & (\ln |(\ln u)'|)' = 2 (\ln \rho )' \quad \Rightarrow \quad (\ln u)' = c\rho^2 
\end{eqnarray*}
for a constant $c$, where we have written $u = \rho /\si$.  This determines $\si$ as a function of $\rho$:
\begin{equation} \label{special-2}
\frac{\rho}{\si} = (1/a) e^{\int c\rho^2 \dd t} \quad \Rightarrow \quad \si = a\rho e^{-\int c\rho^2\dd t}
\end{equation}
for constants $a$ and $c$. It also yields the identity:
$$
(\ln \rho )' - (\ln \si )' = c\rho^2 \quad \Longrightarrow \quad (\ln \si )'' = (\ln \rho )'' - 2c\rho\rho^{\prime}
$$ 
When we combine this with the last identity of \eqref{special-1}, we obtain
\begin{eqnarray}
(\ln \rho )'' + (\ln \rho )'^2 = 2 c \rho\rho^{\prime} \quad \Longrightarrow \quad \rho '' = 2c\rho^2\rho^{\prime} \nonumber \\
 \Longrightarrow \quad \rho^{\prime} = \frac{2c}{3} \rho^3 + e \label{rho}
\end{eqnarray} 
for another constant $e$. Then from  \eqref{special-2}:
\begin{equation} \label{si-rho}
\si =  a\rho e^{-\int c\rho^2\dd t} = a\rho e^{-\int \frac{\rho ''}{\rho '}\dd t} = a\rho e^{-\frac{1}{2} \ln |\rho '|+B } = b\rho |\rho '|^{- 1/2}\,,
\end{equation} 
for constants $B$ and $b$ where $A = \left\{ \begin{array}{rcl} -3b^2e & {\rm if} & \rho'>0 \\ + 3b^2e & {\rm if} & \rho '<0\end{array}\right.$. Conversely, given a solution $\rho$ to \eqref{rho} with $\si$ given by \eqref{special-2}, equations \eqref{proj-1}, \eqref{proj-2} and \eqref{proj-3} are satisfied with $A = \left\{ \begin{array}{rcl} -3b^2e & {\rm if} & \rho'>0 \\ + 3b^2e & {\rm if} & \rho '<0\end{array}\right.$.  Specifically, 
$$
(\ln \si )' = (\ln \rho )' - \tfrac{1}{2}\frac{\rho ''}{\rho '} = (\ln \rho )' - c \rho^2 \ \Longrightarrow \ (\ln \si )'' =( \ln \rho )'' - 2c\rho \rho '\,.
$$
and we now substitute. 

Explicit solutions can be obtained by solving \eqref{rho}.  In the case when $e = 0$, then up to an affine linear change in the $t$ coordinate, the solution is given by $\rho (t) = t^{-1/2}$ with
$$
\sigma (t) = a t^{-1/2}e^{\frac{3}{4}\int t^{-1}\dd t} = at^{1/4}\,.
$$
This corresponds to an incomplete Ricci flat ($A=0$) metric defined on the half-space $t>0$.  

In the case when $e \neq 0$, relabel the constants such that 
\begin{equation} \label{ODE}
\rho ' = \al (\rho^3 - \be^3) = \al (\rho - \be )(\rho^2 + \be \rho + \be^2) \,, \quad (c= 3\al /2 \ {\rm and} \ e = -\al \be^3)\,.
\end{equation} 
Then 
$$
\frac{\dd \rho}{\al (\rho - \be )(\rho^2 + \be \rho + \be^2)} = \dd t\,.
$$
which can be integrated explicitly. 

\begin{figure}
\includegraphics[scale=0.7]{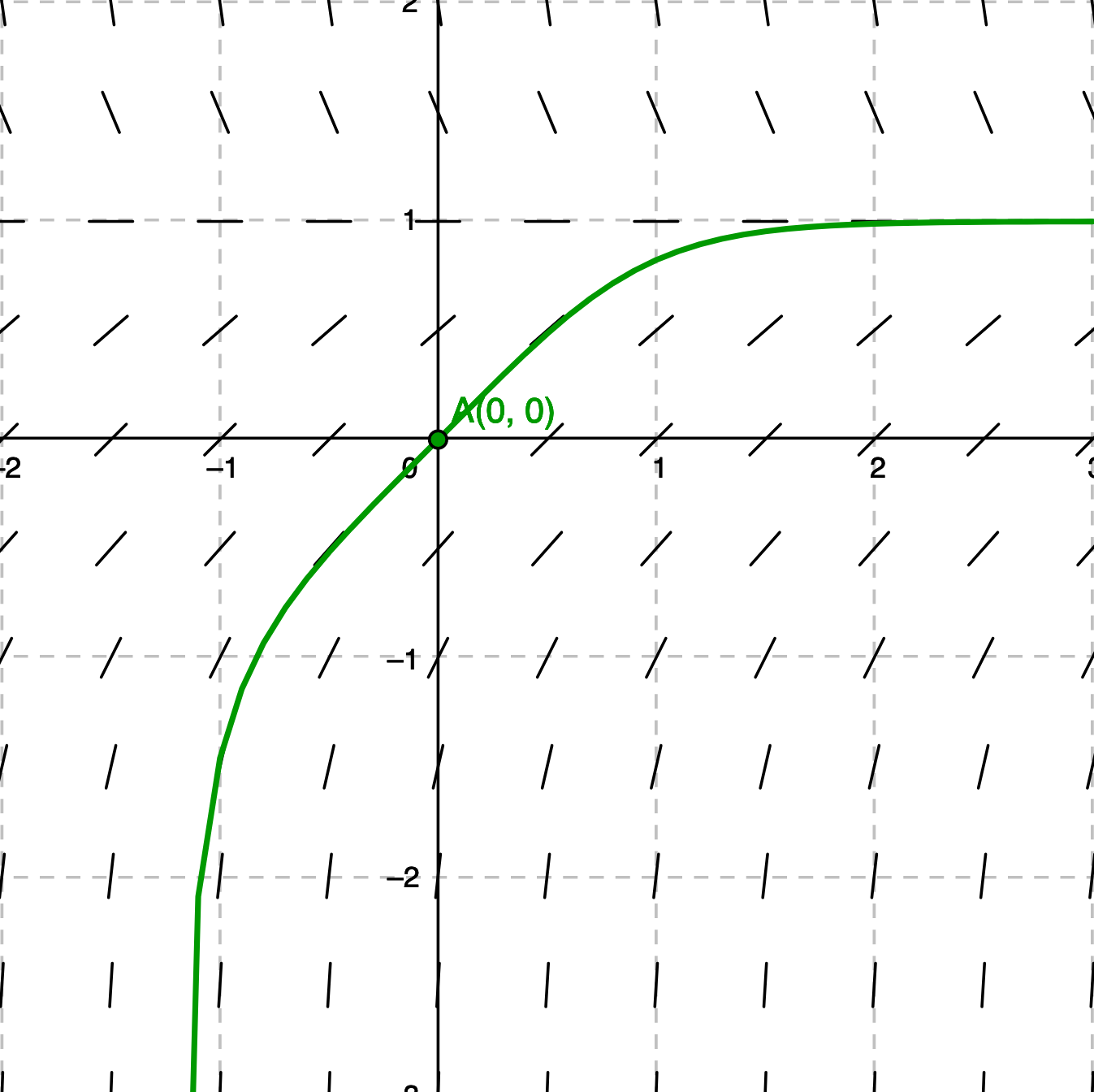}
\caption{Solution to \eqref{ODE} with $\al = -1$, $\be = 1$ and $\rho (0) = 0$} 
\end{figure}

\begin{lemma} \label{lem:ODE} {\rm (i)}  For $\al <0$ and $\be >0$, there is a solution $\rho (t)$ to \eqref{ODE} that exists for all $t\geq 0$, satisfying $\rho (0) = 0$, $\rho '(t)>0$ for all $t \geq 0$ and $0< \rho (t) < \be$ for all $t> 0$.  As $t \ra \infty$, $\rho (t) \ra \be$ and $\rho '(t) \ra 0$.   

\n {\rm (ii)}  For $\al >0$ and $\be <0$, there is a $t_0>0$ and a solution $\rho (t)$ to \eqref{ODE} that exists for all $t\in [0, t_0)$ satisfying $\rho (0) = 0$, $\rho '(t)>0$ for all $t \in [0, t_0)$ and that tends to infinity as $t \ra t_0^-$.   
\end{lemma}

\begin{proof} (i) A solution $\rho (t)$ to \eqref{ODE} in a neighbourhood of $t = 0$ satisfying $\rho (0) = 0$ is guaranteed by the general existence theory of ODEs (see for example \cite{Di} (10.4.5)).  Without loss of generality we can suppose that $\al = -1$ and $\be = 1$ so the equation has the form:
\begin{equation} \label{ODE1}
\rho ' = - \rho^3 + 1\,.
\end{equation} 
  Clearly $\rho '(t)>0$ provided $\rho (t) < 1$. Suppose that $\rho (t)$ achieves the value $1$ and let $t_0>0$ be the first time for which this occurs. Then from \eqref{ODE1}, $\rho '(t_0) = 0$. On differentiating \eqref{ODE1}, we see that $\rho ''(t_0) = - 3 \rho^2(t_0)\rho '(t_0) = 0$, and so on; by recursion all derivatives $\rho ^{(n)} (t_0) = 0$. But by analyticity of the solution (see \cite{Di} (10.5.3)), this means that $\rho (t) \equiv 1$ for all $t$, contradicting the initial condition $\rho (0) = 0$. Thus $\rho (t) < 1$ for all $t \geq 0$.  
    
  Clearly any interval of existence $[0, t_1)$ can be extended to $t \geq t_1$, so the solution exists for all time $t \geq 0$ with $\rho (t) \ra 1$ and $\rho '(t) \ra 0$ as $t \ra \infty$.  
  
  \n (ii) Without loss of generality, suppose that $\al = 1$ and $\be = -1$, so that \eqref{ODE} takes the form
  \begin{equation} \label{ODE2}
  \rho ' = \rho^3+1
  \end{equation}   
 This time we can appeal to the explicit equation determining $\rho$ obtained on integrating \eqref{ODE2} with $\rho (0) = 0$:
 $$
 \frac{1}{3} \ln \frac{\rho +1}{|\rho^2-\rho +1|^{1/2}} + \frac{\sqrt{3}}{3} \arctan \left( \frac{2}{\sqrt{3}}\left(\rho - \frac{1}{2}\right)\right) + \frac{\pi \sqrt{3}}{18} = t\,.
 $$
 Then as $\rho \ra \infty$, the left-hand side $\ra \frac{2\sqrt{3}\pi}{9}$ which yields the upper bound $t_0 = \frac{2\sqrt{3}\pi}{9}$. 
 \end{proof} 

\begin{figure}
\includegraphics[scale=0.6]{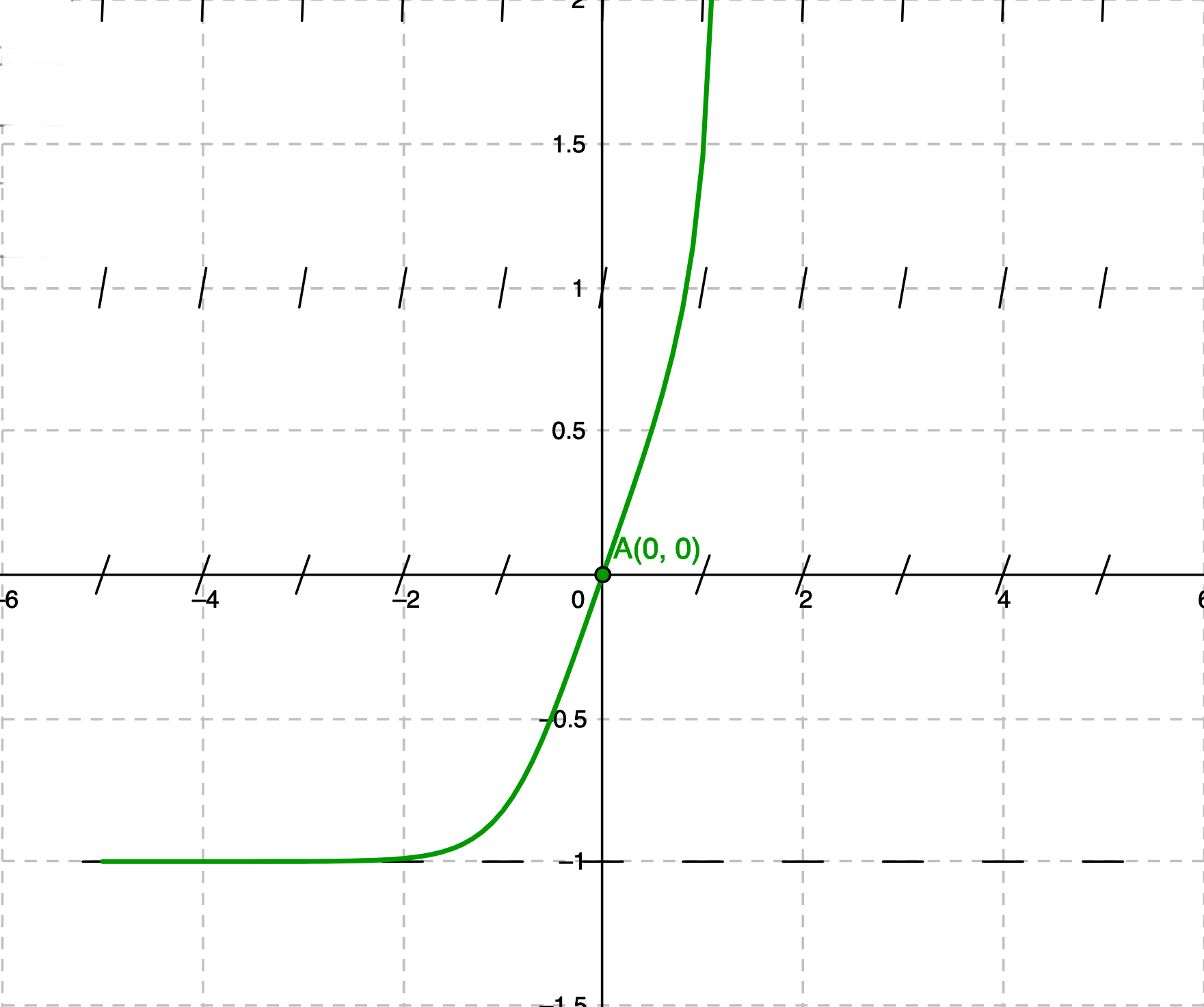}
\caption{Solution to \eqref{ODE} with $\al = 1$, $\be = -1$ and $\rho (0) = 0$} 
\end{figure}

In the following theorem, we consider \emph{ends} as components of the complement of the set $\ve \leq t \leq 1/\ve$ for $\ve$ small. 

\begin{theorem} \label{thm:ends} Solutions to equation \eqref{ODE} yield two families of $4$-dimensional Einstein metrics. Each member of the first family is a complete metric defined on the upper half space $t>0$, having negative Ricci curvature and two ends: one asymtotic to hyperbolic $4$-space and the other to $\RR^2$.  Each member of the second family is incomplete, defined on the space $0<t<t_0$ for a fixed constant $t_0$, and has negative Ricci curvature. 
\end{theorem}

\begin{proof} Consider the solutions to \eqref{ODE} given by Lemma \ref{lem:ODE}(i) and as above, set $e = - \al \be^3>0$.  At $t = 0$, $\rho (0) = 0$,  $\rho '(0) = e$ and $\rho ''(0) = 0$. Thus the Taylor expansion about $t = 0$ has the form $\rho (t) = et + \Oo (t^3)$. For $\si$ we have $\si (0) = 0$ and 
$$
\si = \frac{b\rho}{\sqrt{\rho '}} \ \Longrightarrow \ \si ' = \frac{b (\rho ')^{3/2} - \frac{1}{2}b\rho )\rho ')^{-1/2}\rho ''}{\rho '} \ \Longrightarrow \ \si '(0) = b\sqrt{e}\,,
$$
so that about $t = 0$, we have $\si (t) = b \sqrt{e}\, t + \Oo (t^2)$.  In particular, being of type $(\dd t^2 + \dd x_2{}^2 + \dd x_3{}^2+\dd x_4{}^2)/t^2$, for $t >0$, the metric is complete in a neighbourhood of the boundary $t = 0$.  

The Einstein constant can be deduced from \eqref{proj-3}, \eqref{ODE} and the expression \eqref{si-rho} for $\si$, specifically $A = 3b^2\al \be^3<0$. 

In order to study the ends of the resulting Einstein manifold, we consider the exterior to the set $\ve \leq t \leq 1/\ve$ for $\ve$ small. As $t \ra \infty$, then $\rho (t) \ra \be$, $\rho '(t) \ra 0$ and $\si (t) \ra \infty$. Thus the metric approaches an end of the form $\RR^2$ with metric $(\dd x_3{}^2 + \dd x_4{}^2)/\be^2$.  As already noted, the metric is of hyperbolic type as $t \ra 0^+$.  

A similar analysis takes place for the solutions to \eqref{ODE} given by Lemma \ref{lem:ODE}(ii), but this time $\rho (t) \ra \infty$ as $t \ra t_0^-$, showing the incompleteness of the metric.   
\end{proof}


\begin{thebibliography}{00}

\bibitem{Ba-Da} P. Baird and L Danielo, \emph{Three-dimensional Ricci solitons which project to surfaces}, J. Reine  Angew. Math. {\bf 608} (2007), 65--91.

\bibitem{Ba-Gh} P. Baird and E. Ghandour, \emph{Biconformal equivalence between $3$-dimensional Ricci solitons}, to appear, Tohoku Math. J. 

\bibitem{Ba-Wo} P. Baird and J. C. Wood, Harmonic morphisms between Riemannian manifolds, London Math. Soc Monographs, New Series {\bf 29}, Oxford Univ. Press 2003.  

\bibitem{BCD} A. Bernard, E. Campbell and A. M. Davie, \emph{Brownian motion and generalized analytic and inner functions}, Ann. Inst. Fourier (Grenoble) {\bf 29} (1) (1979), 207--228.  

\bibitem{Be} A. Besse, Einstein Manifolds, Springer-Verlag, 1987.  

\bibitem{Da-1} L. Danielo.  Structures Conformes, Harmonicit\'e et M\'etriques d'Einstein.  Thesis.  Universit\'e de Bretagne Occidentale, 2004.  

\bibitem{Da-2} L. Danielo, \emph{Construction de m\'etriques d'Einstein \`a partir de transformations biconformes}, Ann. Fac. des Sciences de Toulouse, S\'er. 6, {\bf 15}, no. 3 (2006), 553-588.  

\bibitem{Di} J. Dieudonn\'e, Foundations of Modern Analysis, Academic Press, 1969. 

\bibitem{Fu}  B. Fuglede, \emph{Harmonic morphisms between Riemannian manifolds}, Ann. Inst. Fourier (Grenoble) {\bf 28} (2) (1978), 107--144. 


\bibitem{Gh} E. Ghandour, Applications semi-conformes et solitons de Ricci, Thesis, Universit\'e de Bretagne Occidentale, 2018. 

\bibitem{Gr} M. Gromov, \emph{Pseudo holomorphic curves in symplectic manifolds}, Invent. Math. {\bf 82} (1985), 307-347. 

\bibitem{He} E. Hebey, Introduction \`a l'analyse non lin\'eaire sur les vari\'et\'es. Diderot, Paris, 1997. 

\bibitem{He2} E. Hebey, Scalar curvature type problems in Riemannian geometry. Notes of a course given at the University of Rome 3. http://www.u-cergy.fr/rech/pages/hebey/  

\bibitem{Hi} D. Hilbert, Die Grundlagen der Physik. Nachr. Ges. Wiss. G\"ottingen, (1915, 395-407.  

\bibitem{Hi} N. J. Hitchin, \emph{On compact four-dimensional Einstein manifolds}, J. Diff. Geom. {\bf 9} (1974), 435-442.    

\bibitem{Is}  T. Ishihara, \emph{A mapping of Riemannian manifolds which preserves harmonic functions}, J. Math. Kyoto Univ. {\bf 19} (1979), 215--229.  

\bibitem{Ka-Wa} J. Kazdan and F. Warner, \emph{Scalar curvature and conformal deformation of Riemannian structure}, J. Differential Geom. {\bf 10} (1975), 113-134.  


\bibitem{ON} B. O'Neill, \emph{The fundamental equations of a submersion}, Michigan Math. J., {\bf 13} (1966), 459-469.   

\bibitem{Pa} R. Pantilie, \emph{Harmonic morphisms with $1$-dimensional fibres on $4$-dimensional Einstein manifolds}, Comm. Anal. Geom., {\bf 10} (2002), 779-814.  

\bibitem{Sp} M. Spivak, A Comprehensive Introduction to Riemannian Geometry, 2nd edn. Publish of Perish, Wilmongton DE, 1979.  

\bibitem{Sc} R. Schoen, \emph{Conformal deformation of a Riemannian metric to constant scalar curvature}, J. Diff. Geom. {\bf 20}, (1984), 479-495.  

\bibitem{Th} J. A. Thorpe, \emph{Some remarks on the Gauss-Bonnet formula}, J. Math. Mech. {\bf 18}, (1969), 779-786.

\bibitem{Vi} M. Ville, {Harmonic morphisms from Einstein $4$-manifolds to Riemann surfaces}, Internat. J. Math., {\bf 14} (2003), 327-337. 

\bibitem{Wo} J. C. Wood, \emph{Harmonic morphisms and Hermitian structures on Einstein $4$-manifolds}, Internat. J. Math., {\bf 3} (1992), 415-439. 

\bibitem{Ya} H. Yamabe, \emph{On a deformation of Riemannian structures on compact manifolds}, Osaka Math. J., {\bf 12} (1960), 21-37.  


\end{thebibliography}
\end{document}